\documentclass[a4paper,11pt]{article}
\usepackage[all]{xy}
\usepackage[T1]{fontenc}
\usepackage[utf8]{inputenc}
\usepackage{lmodern}
\usepackage{microtype}
\usepackage{mathtools}
\usepackage{amsfonts}
\usepackage{amsmath}
\usepackage{amssymb}
\usepackage{pictexwd,dcpic}
\usepackage{thmtools}
\usepackage[top=3cm,bottom=3cm,left=2.5cm,right=2.5cm]{geometry} % Adjusted margins for better appeal
\usepackage{hyperref}
\usepackage{fancyhdr}
\usepackage{sectsty} % For customizing section titles
\usepackage{titlesec} % For fine-tuning section spacing
\usepackage{parskip} % For paragraph spacing

\sectionfont{\large\bfseries\centering}
\subsectionfont{\normalsize\bfseries}
\titleformat{\section}{\large\bfseries\centering}{\thesection}{1em}{}
\titlespacing{\section}{0pt}{2em plus 1em minus 0.5em}{1em plus 0.5em minus 0.5em}
\titlespacing{\subsection}{0pt}{1.5em plus 0.5em minus 0.5em}{0.75em plus 0.25em minus 0.25em}

\newtheorem{theorem}{Theorem}[section]
\newtheorem{lemma}[theorem]{Lemma}
\newtheorem{proposition}[theorem]{Proposition}
\newtheorem{corollary}[theorem]{Corollary}
\newtheorem{definition}[theorem]{Definition}
\newtheorem{rem}[theorem]{Remark}
\newenvironment{proof}[1][Proof]{\begin{trivlist}
\item[\hskip \labelsep {\bfseries #1}]}{\end{trivlist}}

\newenvironment{conjecture}[1][Conjecture]{\begin{trivlist}
\item[\hskip \labelsep {\bfseries #1}]}{\end{trivlist}}

\newcommand{\qed}{\nobreak \ifvmode \relax \else
      \ifdim\lastskip<1.5em \hskip-\lastskip
      \hskip1.5em plus0em minus0.5em \fi \nobreak
      \vrule height0.75em width0.5em depth0.25em\fi}

\DeclareFontFamily{U}{wncy}{}
\DeclareFontShape{U}{wncy}{m}{n}{<->wncyr10}{}
\DeclareSymbolFont{mcy}{U}{wncy}{m}{n}
\DeclareMathSymbol{\Sh}{\mathord}{mcy}{"58}

\date{}
\title{\textbf{%Equivariant Leading Terms and 
Class Number Relations in Abelian Extensions of Global Fields}}
\author{
  Saad EL BOUKHARI\hspace*{0.2 cm}and\hspace*{0.2 cm}Ben FORR\'AS
   }
\newcommand{\Sc}{\Sigma}
\begin{document}
\maketitle
\footnotetext[1]{S. El Boukhari. University Marie et Louis Pasteur, Lab. of Math. of Besançon, Besançon, 25000, France. 
\textit{E-mails:}  \texttt{saad.el\_boukhari@umlp.fr}, \texttt{saadelboukhari1234@gmail.com}}
\footnotetext[2]{B. Forrás. University of Ottawa, Department of Mathematics and Statistics, STEM Complex, 150 Louis Pasteur Pvt, Ottawa, ON, Canada K1N 6N5. \textit{E-mails:}  \texttt{bforras@uottawa.ca}}
\vspace{3em} % Increased space after title for a cleaner look

\begin{abstract}
Consider a finite abelian extension $K/k$ of global fields with Galois group \(G\). We study the rank-one component of the generalized Stickelberger module
associated with \(K/k\) and a finite set \(S\) of places of \(k\). Under explicit
splitting conditions on \(S\), we compute generalized indices of this module with respect to the torsion-free part of the \(S\)-unit group of \(K\). We also obtain \(p\)-primary refinements which include the non-semisimple case \(p\mid |G|\). As applications, we derive divisibility relations between the
\(S\)-class numbers of \(K\) and \(k\), both for number fields and for function fields.
\end{abstract}

\vspace{2em}

\textit{Keywords: Global fields, Artin $L$-functions, generalized index of modules, class numbers, Dirichlet regulator, group homology and cohomology, Rubin--Stark elements.}

\vspace{1em}

\textit{2020 Mathematics Subject Classification: 11R42, 11R29, 11R27, 11R59.}

\vspace{3em} 
\section{Introduction}

Let $K/k$ be a finite abelian extension of global fields, and let $S$ be a finite set of places of $k$. 
In this article, we prove divisibility relations between the $S$-class numbers of $k$ and $K$ under certain hypotheses.
This objective is achieved by computing the index of the so-called generalized Stickelberger module $\varrho_{K/k,S}$.
This is a certain submodule of the $\mathbb C$-vector space $\mathbb C \otimes_{\mathbb Z} O^\times_{K,S}$, where $O^\times_{K,S}$ is the group of $S$-units in $K$. The generalized Stickelberger module $\varrho_{K/k,S}$ is defined in terms of the leading term of the equivariant $S$-depleted Artin $L$-function of $K/k$ at zero: thereby $\varrho_{K/k,S}$ is connected to $S$-class numbers via the analytic class number formula.

For $r\ge0$, the rank-$r$ Rubin--Stark conjecture, originally formulated for number fields by Rubin in the seminal paper \cite{Rubin}, predicts the existence of certain integral elements related to the $r$th derivative of the equivariant $S$-depleted Artin $L$-function. 
For characters of $G$ for which the $S$-depleted Artin $L$-function has vanishing order one, the generalized Stickelberger module $\varrho_{K/k,S}$ is closely related to rank-one Rubin--Stark elements.
The Rubin--Stark conjecture was extended to function fields by Popescu \cite{PopescuStark}. The conjecture is known in many cases of interest: for instance, by a theorem of Burns \cite{BurnsCongruences01}, it holds whenever the equivariant Tamagawa number conjecture holds. See Remark~\ref{REMA_3_2} for a list of cases in which the Rubin--Stark conjecture has been proven.

We now describe our results. Letting the depletion set $S$ consist of a single place completely split in $K/k$, we compute the $p$-part of the generalized index of the generalized Stickelberger module inside the torsion-free part of $O^\times_{K,S}$ for each rational prime $p$ (Theorem~\ref{Theo:p-primary-rank-one-index-formula}). This index is given by the quotient of the $S$-class numbers multiplied by a certain local factor at $p$. We provide a bound for this factor in general, and show that it is trivial in several important cases.
Using the rank one index formula for the generalized Stickelberger module, we also derive divisibility relations between the $S$-class groups under the assumption of the Rubin--Stark conjecture (Theorem~\ref{Thm:RS-obstruction-divisibility}).
These results constitute generalizations of the main theorems of \cite{El26}, where $k$ is taken to be a quadratic imaginary number field and $S$ is the set consisting of the single archimedean place, and the degree $[K:k]$ is assumed to be coprime to $p$.

We also determine the index of the generalized Stickelberger module for depletion sets $S$ containing more than one place. In particular, in the case when $S$ consists of two places, one completely split and one that does not decompose in $K/k$, we obtain divisibility relations for $S$-class numbers (Theorem~\ref{Thm:cardinality-two-final-divisibility}).

This article is organised as follows. Standing notations are defined in \S\ref{section:notations}. In \S\ref{section:preliminaries}, we recall the relevant results concerning the Rubin--Stark conjecture and generalized Stickelberger modules. The novel part of this work is contained in \S\ref{section:rank-one}: after making some general observations, we distinguish between cases depending on the cardinality of the depletion set $S$ in ascending order. Finally in \S\ref{section:conclusion}, we discuss the problem of obtaining divisibility relations for $S$-class groups when $|S|\ge3$.

\paragraph{Acknowledgement}
BF's research is funded by the Deutsche Forschungsgemeinschaft (DFG, German Research Foundation), project no.~559516518.

\section{Notations} \label{section:notations}

We fix the following notations for the remainder of this work.

\begin{itemize}
\item Let \(K/k\) be a finite abelian extension of global fields, and put
\(
G:=\mathrm{Gal}(K/k).
\)

\item If \(k\) is a number field, let \(S_\infty\) denote the set of archimedean
places of \(k\). If \(k\) is a function field, we put \(S_\infty=\emptyset\).
We denote by \({\mathrm{Ram}}(K/k)\) the set of places of \(k\) which ramify
in \(K/k\).

\item If \(S\) is a finite set of places of \(k\), we shall always assume that
\(
S\supseteq S_\infty
\)
in the number field case, and that \(S\neq\emptyset\) in the function field case.

\item For such a set \(S\), let \(O_{K,S}\) denote the ring of \(S\)-integers of
\(K\), and let
\(
O_{K,S}^{\times}
\)
denote its group of units.

\item If $L$ is a global field, let \(\mu(L)\) denote the group of roots of unity in \(L\), and put
\(
w_L:=|\mu(L)|.
\)

\item Let
\(
\widehat{G}:=\mathrm{Hom}(G,\mathbb{C}^{\times})
\)
be the group of complex characters of \(G\). For \(\chi\in\widehat{G}\), we write
\[
e_\chi:=\frac{1}{|G|}\sum_{\sigma\in G}\chi(\sigma)\sigma^{-1}
\]
for the associated idempotent in \(\mathbb{C}[G]\). We denote by \(\chi_0\) the
trivial character of \(G\).

\item If \(v\) is a place of \(k\), we write \(D_v=D_v(K/k)\) for the decomposition
subgroup of \(G\) at \(v\). This is well-defined up to conjugacy, and hence is
well-defined as a subgroup of \(G\), since \(G\) is abelian. If $D_v=G$, then we say that \(v\) does not decompose in \(K/k\).

\item If \(S\) is a finite set of places of \(k\), then \(S_K\) denotes the set of
places of \(K\) above \(S\). We let \(Y_{S,K}:=\mathbb{Z}S_K\) denote the free abelian
group on \(S_K\), and define
\[
X_{K,S}:=
\left\{
\sum_{w\in S_K}a_w w\in \mathbb{Z}S_K
\;\middle|\;
\sum_{w\in S_K}a_w=0
\right\}.
\]

\item If \(v\) is a non-archimedean place of a global field, we write
\[
Nv:=|\kappa(v)|
\]
for the cardinality of its residue field. We do not use this notation for
archimedean places.

\item We normalize absolute values so that the product formula holds. Thus, if
\(w\) is non-archimedean, then
\[
|x|_w=Nw^{-\mathrm{ord}_w(x)}.
\]
In the number field case, at archimedean places we use the usual product-formula
normalization.

\item If \(M\) is a \(\mathbb{Z}\)-module of finite type, let \(M_{\mathrm{tors}}\)
denote its torsion \(\mathbb{Z}\)-submodule, and let
\(
M_{\mathrm{tf}}:=M/M_{\mathrm{tors}}
\)
denote its torsion-free quotient. If \(R\) is any commutative ring, we write
\[
RM:=R\otimes_{\mathbb{Z}}M.
\]
If $M$ is finite, we may write $M(p):=\mathbb{Z}_pM$ for the $p$-part of $M$.

\item For a global field \(L\), and for a finite non-empty set \(S\) of places
of \(L\) containing the archimedean places in the number field case, we define
the logarithmic regulator map by
\[
\lambda_{L,S}:O_{L,S}^{\times}\longrightarrow \mathbb R X_{L,S},
\qquad
u\longmapsto -\sum_{w\in S}\log |u|_w\,w.
\]
It has finite kernel and its image is a full lattice in \(\mathbb R X_{L,S}\).
We denote by \(R_{L,S}\) the covolume of
\(\lambda_{L,S}((O_{L,S}^{\times})_{\mathrm{tf}})\) with respect to the lattice
\(X_{L,S}\). If \(O_{L,S}^{\times}\) has rank zero, we use the convention
\[
R_{L,S}=1.
\]

\item For a global field \(L\), and for a non-empty finite set of places \(S\) of \(L\), we
write \(\mathrm{Cl}_{L,S}\) for the \(S\)-class group of \(L\), and
\(
h_{L,S}:=|\mathrm{Cl}_{L,S}|.
\) If $p$ is a prime number we also write: $h_{L,S}(p):=p^{v_p(h_{L,S})}.$

\item For a global field \(L\), we write \(\zeta_{L,S}(s)\) for the
\(S\)-truncated zeta function of \(L\). We always take leading terms at
\(s=0\) with respect to the variable \(s\). Thus
\[
\zeta_{L,S}^{*}(0):=
\lim_{s\to 0}s^{-\mathrm{ord}_{s=0}\zeta_{L,S}(s)}\zeta_{L,S}(s).
\]
With the above normalization of absolute values and regulators, the
\(S\)-class number formula has the uniform form
\[
\zeta_{L,S}^{*}(0)=\pm\frac{h_{L,S}R_{L,S}}{w_L}
\]
for both number fields and function fields.

\item If \(p\) is a rational prime number, we denote by \(v_p\) the associated
\(p\)-adic valuation.

\item If \(R\) is a commutative ring with unity and \(H\) is a finite group, we write
\[
\mathrm{Aug}_H:R[H]\longrightarrow R,
\qquad
\sum_{\sigma\in H}a_\sigma\sigma\longmapsto
\sum_{\sigma\in H}a_\sigma
\]
for the augmentation map relative to \(H\). We write
\[
\Delta_R H:=\ker(\mathrm{Aug}_H)
=
\left\langle \sigma-1\;\middle|\;\sigma\in H\right\rangle_{R[H]}
\]
for the associated augmentation ideal. If \(R=\mathbb{Z}\), we simply write
\(
\Delta H:=\Delta_{\mathbb{Z}}H.
\)
If \(R=\mathbb{Z}_p\), we write
\(
\Delta_p H:=\Delta_{\mathbb{Z}_p}H.
\)
\end{itemize}
\section{Preliminaries} \label{section:preliminaries}
\subsection{Equivariant leading terms at \(s=0\)}

Let \(K/k\) be a finite abelian extension of global fields with
\[
G:=\mathrm{Gal}(K/k).
\]
Let \(S\) be a finite set of places of \(k\). In the number field case, we
assume that \(S\) contains all archimedean places of \(k\); in the function
field case, we assume that \(S\) is non-empty. We do not assume, unless
explicitly stated, that \(S\) contains the places ramified in \(K/k\).

For \(\chi\in\widehat G\), let \(L_S(s,\chi)\) be the usual \(S\)-truncated
Artin \(L\)-function associated with \(\chi\), with the standard Artin local
factors at the places \(v\notin S\). In particular, at ramified places
\(v\notin S\), the local factor is defined using inertia invariants. Thus the
definition does not require \(S\) to contain \({\mathrm{Ram}}(K/k)\). For more details, see %\cite[Rem. 3.1]{Solomon} 
\cite[Ch. VII, \S 10]{Neukirch} or \cite[Chs. 0,1]{Tate84}.

For \(v\in S\), let \(D_v=D_v(K/k)\) denote the decomposition subgroup of \(G\)
at \(v\). We put
\[
r_S(\chi):=\mathrm{ord}_{s=0}L_S(s,\chi).
\]
The standard order formula remains valid in this imprimitive setting:
\begin{equation} \label{order-formula}
r_S(\chi)=
\begin{cases}
\left|\left\{v\in S\mid \chi(D_v)=1\right\}\right|,
& \text{if }\chi\neq \chi_0,\\[0.4em]
|S|-1,
& \text{if }\chi=\chi_0,
\end{cases}
\end{equation}
where \(\chi_0\) denotes the trivial character of \(G\).
\subsection{Relative indices of lattices}

We recall the elementary notion of generalized index that will be used throughout
the paper. Let \(O\) be either \(\mathbb{Z}\) or \(\mathbb{Z}_p\), and let \(E\)
be its field of fractions, namely \(E=\mathbb{Q}\) or \(E=\mathbb{Q}_p\),
respectively. More generally, one may also work after extending scalars to
\(\mathbb{C}\) or \(\mathbb{C}_p\), depending on the chosen realization.

Let \(V\) be a finite-dimensional \(E\)-vector space. By a full \(O\)-lattice in
\(V\), we mean a finitely generated free \(O\)-submodule \(M\subset V\) such that
the natural map
\[
E\otimes_O M\longrightarrow V
\]
is an isomorphism.
If \(M\) and \(N\) are two full \(O\)-lattices in \(V\), their generalized index
is defined as follows. Choose an \(E\)-linear automorphism
\(
\alpha:V\longrightarrow V
\)
such that
\[
\alpha(M)=N.
\]
Then set
\[
(M:N)_O:=\mathrm {det}_E(\alpha)\,O.
\]
This %fractional ideal of \(O\) 
module does not depend on the choice of \(\alpha\). Thus
\((M:N)_O\) gives a well-defined measure of the relative position of the two
lattices.
When \(N\subseteq M\), this definition recovers the ordinary index:
\[
(M:N)_O=[M:N]\,O.
\]
We shall use this notation both in the classical situation, where one lattice is
contained in the other, and in the more general situation where no containment is
assumed.  For a detailed exposition on generalized indices and their properties, see for example \cite[\S 3]{SaadMazigh}.

\subsection{Higher special elements for global fields and the Rubin--Stark conjecture}
\label{Subsec:Rubin--Stark-global-fields}

We recall the Rubin--Stark elements in the setting needed below, as introduced by Rubin \cite[\S2]{Rubin} for number fields and Popescu \cite[\S1]{PopescuStark} for function fields.
Let \(K/k\) be
a finite abelian extension of global fields with Galois group
\[
G:=\mathrm{Gal}(K/k).
\]
Let \(S\) be a finite set of places of \(k\). In the number field case, we assume
that \(S\) contains the set \(S_\infty\) of archimedean places of \(k\). In the
function field case, we assume that \(S\) is non-empty. We state a more general form of the Rubin--Stark conjecture here. For this, we do not assume here
 that \(S\) necessarily contains the places of \(k\) which ramify in \(K/k\).

The logarithmic regulator is defined as
\[
\lambda_{K,S}:O_{K,S}^{\times}\longrightarrow \mathbb{R}X_{K,S},
\qquad
u\longmapsto -\sum_{w\in S_K}\log |u|_w\,w.
\]
It has finite kernel and its image is a full lattice in \(\mathbb{R}X_{K,S}\).
Therefore, after tensoring with \(\mathbb{R}\), it induces an isomorphism for any positive integer $r\in\mathbb{Z}_{\geq 0}$
\[
\widetilde{\lambda}_{K,S}:
\mathbb{R}\bigwedge_{\mathbb{Z}[G]}^r O_{K,S}^{\times}
\xrightarrow{\;\simeq\;}
\mathbb{R}\bigwedge_{\mathbb{Z}[G]}^r X_{K,S}.
\]

Let \(T\) be a finite set of places of \(k\), disjoint from \(S\) and from $\mathrm{Ram}(K/k)$. 
 Let \(O_{K,S,T}^{\times}\) be the group of \(S\)-units congruent to \(1\) at all places
above \(T\).
Since \(O_{K,S,T}^{\times}\) has finite
index in \(O_{K,S}^{\times}\), the inclusion
\(
O_{K,S,T}^{\times}\hookrightarrow O_{K,S}^{\times}
\)
induces a natural identification
\[
\mathbb{R}\bigwedge_{\mathbb{Z}[G]}^r O_{K,S,T}^{\times}
=
\mathbb{R}\bigwedge_{\mathbb{Z}[G]}^r O_{K,S}^{\times}.
\]
For \(\chi\in\widehat{G}\), let
\[
L_{S,T}(s,\chi):=L_{S}(s,\chi)\prod_{v\in T}(1-\chi(\sigma_v)Nv^{1-s})
\]
and observe that 
 the order of vanishing at \(s=0\) of \(L_{S,T}(s,\chi)\)  is independent of the auxiliary
set \(T\) and is equal to $r_S(\chi)$. 
We now fix a subset
\[
V=\{v_1,\dots,v_r\}\subset S
\]
consisting of places which split completely in \(K/k\). We also choose an
auxiliary place
\(
v_0\in S\setminus V,
\)
and write
\[
S=\{v_0,v_1,\dots,v_t\}.
\]
For each \(0\leq i\leq t\), fix a place \(w_i\) of \(K\) above \(v_i\).
Since every place in \(V\) splits completely in \(K/k\), it follows from \eqref{order-formula} that
\[
r_S(\chi)\geq r
\qquad
\text{for all }\chi\in\widehat{G}.
\]
We define the equivariant \(S\)-truncated Artin \(L\)-function by
\[
\Theta_{K/k,S,T}(s)
:=
\sum_{\chi\in\widehat{G}}L_{S,T}(s,\chi^{-1})e_\chi
\in \mathbb{C}[G].
\]
For every \(\chi\in\widehat{G}\), set
\[
L_{S,T}^*(0,\chi)
:=
\lim_{s\to 0}s^{-r_S(\chi)}L_{S,T}(s,\chi).
\]
The leading term of \(\Theta_{K/k,S}(s)\) at \(s=0\) is then
\[
\Theta_{K/k,S,T}^*(0)
:=
\sum_{\chi\in\widehat{G}}L_{S,T}^*(0,\chi^{-1})e_\chi.
\]
Let
\[
\Theta_{K/k,S,T}^{(r)}(0)
:=
\sum_{\chi\in\widehat{G}}
e_\chi
\lim_{s\to 0}s^{-r}L_{S,T}(s,\chi^{-1})
\in \mathbb{C}[G].
\]
Equivalently, if
\[
e_S^r:=
\sum_{\substack{\chi\in\widehat{G}\\ r_S(\chi)=r}}e_\chi,
\]
then
\[
\Theta_{K/k,S,T}^{(r)}(0)
=
e_S^r\Theta_{K/k,S,T}^*(0).
\]

\begin{definition}
The \(r\)-th order Rubin--Stark element associated with the data
\((K/k,S,T,V)\) is the element
\[
\eta_{K/k,S,T}^r
\in
\mathbb{R}\bigwedge_{\mathbb{Z}[G]}^r O_{K,S,T}^{\times}
\]
defined by
\[
\widetilde{\lambda}_{K,S}\bigl(\eta_{K/k,S,T}^r\bigr)
=
\Theta_{K/k,S,T}^{(r)}(0)\cdot
\bigwedge_{i=1}^r(w_i-w_0).
\]
\end{definition}

The element \(\eta_{K/k,S,T}^r\) is independent of the auxiliary choices of
\(v_0\) and \(w_0\), after the usual identifications.
We recall the definition of Rubin's lattice, also known as the exterior power bidual, associated with
\(O_{K,S,T}^{\times}\):
\[
\bigcap_{\mathbb{Z}[G]}^r O_{K,S,T}^{\times}
:=\Big(\bigwedge_{\mathbb{Z}[G]}^r (O_{K,S,T}^{\times})^*\Big)^*\simeq
\left\{
a\in
\mathbb{Q}\otimes_{\mathbb{Z}}
\bigwedge_{\mathbb{Z}[G]}^r O_{K,S,T}^{\times}
\;\middle|\;
\Phi(a)\in \mathbb{Z}[G]
\text{ for all }
\Phi\in
\bigwedge_{\mathbb{Z}[G]}^r (O_{K,S,T}^{\times})^*
\right\},
\]
where, for any $\mathbb{Z}[G]$-module $M$:
\[
M^*
:=
\mathrm{Hom}_{\mathbb{Z}[G]}
\left(
M,
\mathbb{Z}[G]
\right).
\]
In particular, in ranks \(0\) and \(1\) one has the natural identifications
\[
\bigcap_{\mathbb{Z}[G]}^0 O_{K,S,T}^{\times}
=
\mathbb{Z}[G],
\qquad
\bigcap_{\mathbb{Z}[G]}^1 O_{K,S,T}^{\times}
\cong
(O_{K,S,T}^{\times})_{\mathrm{tf}}.
\]
\begin{conjecture}[Rubin--Stark conjecture]
\textit{The Rubin--Stark element belongs to the associated Rubin's lattice:}
\[
\eta_{K/k,S,T}^r
\in
\bigcap_{\mathbb{Z}[G]}^r O_{K,S,T}^{\times}.
\]
\end{conjecture}
\begin{rem}\label{REMA_3_2}
If $S\supset \mathrm{Ram}(K/k)$, then the Rubin--Stark conjecture is known in several important cases. 
We recall some of them for context.
In the number field case, we further assume that $O^\times_{K,S,T}$ is torsion-free; this condition holds whenever $T$ contains two primes of different residue characteristics or a single prime of large enough absolute norm.

\begin{itemize}
    \item For global function fields, the conjecture is a theorem of Burns; more
    precisely, it follows from \cite[Thm.~A]{BurnsCongruences01}.

    \item In the number field case, the conjecture is known for arbitrary values
    of \(r\) when \(K\) is a finite abelian extension of \(\mathbb{Q}\) by \cite[Thm.~A]{BurnsCongruences01}. It is also known for certain abelian extensions
    of imaginary quadratic fields, through the validity of the equivariant
    Tamagawa number conjecture in these settings; see for instance
    \cite[Theorem~B]{BullachHofer}.

    \item In rank one $r=1$, the Rubin--Stark conjecture holds for finite abelian extensions of imaginary
    quadratic fields. In this case Rubin--Stark elements can be defined through elliptic units; see
    \cite[Chap.~IV, Prop.~3.9]{Tate84}.

    \item In the rank-zero CM setting, the relevant integrality statement is
    closely related to the work of Deligne and Ribet \cite{DR}.

    \item Further known cases include arbitrary quadratic extensions for all
    values of \(r\), as well as a large class of multiquadratic extensions in
    the rank-one case; see for example \cite[\S 3.1]{Buckingham}.
\end{itemize}
\end{rem}
In what follows, we will say that the rank-$r$ Rubin--Stark conjecture holds for $K/k$ and $S$ if we have
\(
\eta_{K/k,S,T}^r
\in
\bigcap_{\mathbb{Z}[G]}^r O_{K,S,T}^{\times}
\) for $K/k$, $S$, $r$ and all finite sets $T$.

Now, let \(K/k\) be a finite abelian extension of global fields, and let
\((S,T,V)\) be as above with
\[
V=\{v_1\}.
\]
\begin{proposition}
\label{Prop:RS-integrality-rank-one}
Assume that the Rubin--Stark conjecture holds for the data
\((K/k,S,T)\). Then the rank-one Rubin--Stark element
\(\eta_{K/k,S,T}^1\) satisfies
\[
\eta_{K/k,S,T}^1
\in
(O_{K,S,T}^{\times})_{\mathrm{tf}}
\subseteq
(O_{K,S}^{\times})_{\mathrm{tf}}.
\]
\end{proposition}

\begin{proof}
By the Rubin--Stark conjecture for the data \((K/k,S,T,V)\), one has
\[
\eta_{K/k,S,T}^1
\in
\bigcap_{\mathbb Z[G]}^1 O_{K,S,T}^{\times}.
\]
By definition, in rank one,
\[
\bigcap_{\mathbb Z[G]}^1 O_{K,S,T}^{\times}
=
\left(O_{K,S,T}^{\times}\right)^{**},
\]
where the dual is taken over \(\mathbb Z[G]\). 
Observe that the ring
\(\mathbb Z[G]\) is a Gorenstein order, see e.g.~\cite[p.~779]{CR1}, and it has Krull dimension $1$. 
The \(\mathbb Z[G]\)-lattice
\((O_{K,S,T}^{\times})_{\mathrm{tf}}\) is $\mathbb{Z}$-torsion-free, and therefore \(\mathbb{Z}[G]\)-reflexive by \cite[p.~1349]{Vasconcelos}. Hence the natural
map
\[
(O_{K,S,T}^{\times})_{\mathrm{tf}}
\longrightarrow
\left((O_{K,S,T}^{\times})_{\mathrm{tf}}\right)^{**}
\]
is an isomorphism. Moreover, since \(\mathbb Z[G]\) is \(\mathbb Z\)-torsion-free,
taking \(\mathbb Z[G]\)-duals kills the \(\mathbb Z\)-torsion of
\(O_{K,S,T}^{\times}\), and hence
\[
\left(O_{K,S,T}^{\times}\right)^{**}
=
\left((O_{K,S,T}^{\times})_{\mathrm{tf}}\right)^{**}.
\]
It follows that
\[
\bigcap_{\mathbb Z[G]}^1 O_{K,S,T}^{\times}
=
(O_{K,S,T}^{\times})_{\mathrm{tf}}.
\]
Therefore
\[
\eta_{K/k,S,T}^1
\in
(O_{K,S,T}^{\times})_{\mathrm{tf}}.
\]

Finally, the inclusion
\[
O_{K,S,T}^{\times}\subseteq O_{K,S}^{\times}
\]
induces an injection on torsion-free quotients
\[
(O_{K,S,T}^{\times})_{\mathrm{tf}}
\hookrightarrow
(O_{K,S}^{\times})_{\mathrm{tf}}.
\]
This proves the proposition.\qed
\end{proof}

\subsection{Rank-one generalized Stickelberger modules}
The $S$-logarithmic regulator
\[
\lambda_{K,S}:O_{K,S}^{\times}\longrightarrow \mathbb{R}X_{K,S},
\qquad
u\longmapsto -\sum_{w\in S_K}\log |u|_w\,w.
\]
induces, 
after extension of scalars,  a
\(\mathbb{C}[G]\)-linear isomorphism
\[
\widetilde{\lambda}_{K,S}:
\mathbb{C} O_{K,S}^{\times}
\xrightarrow{\;\simeq\;}
\mathbb{C} X_{K,S}.
\]
Explicitly,
\[
\widetilde{\lambda}_{K,S}(z\otimes u)
=
z\otimes \lambda_{K,S}(u),
\qquad
z\in \mathbb{C},\quad u\in O_{K,S}^{\times}.
\]

We consider the following module
\[
\vartheta_{K/k,S}(0)
:=
\mathrm{Ann}_{\mathbb{Z}[G]}(\mu(K))\,
\Theta_{K/k,S}^{*}(0)
\]
and then define the
generalized Stickelberger module by
\[
\varrho_{K/k,S}
:=
\widetilde{\lambda}_{K,S}^{-1}
\left(
\vartheta_{K/k,S}(0)X_{K,S}
\right)
\subseteq
\mathbb{C}O_{K,S}^{\times}.
\]

Equivalently,
\[
\varrho_{K/k,S}
=
\left\{
x\in \mathbb{C} O_{K,S}^{\times}
\;\middle|\;
\widetilde{\lambda}_{K,S}(x)
\in
\vartheta_{K/k,S}(0)X_{K,S}
\right\}.
\]

Let \(p\) be a rational prime. For each such prime $p$, we fix once and for all an isomormhism
\[\mathbb{C}\simeq \mathbb{C}_p\]
through which we can see all the previously defined objects $p$-adically.
For our next result, let $S_p$ denote set of places of $k$ comprised of the places in $S$, the places in ${\mathrm{Ram}}(K/k)$ and the places above $p$.
If $T$ is any finite set of places of $k$ disjoint from the set of places which ramify in $K/k$, then we set
\[\delta_{K/k,T}:=\prod_{v\in T}(1-\sigma_v^{-1}\mathrm{N}v)\]
where $\sigma_v$ is the Frobenius automorphism in $G$. The rank-one Rubin--Stark conjecture provides the following description of the generalized Stickelberger module.
\begin{lemma}
\label{Lem:rho-generated-by-rank-one-RS-elements}
Assume that the rank-one Rubin--Stark conjecture holds for $K/k$ and $S$.
Then
\[
e_S^1\mathbb Z_p\varrho_{K/k,S}
\subseteq
e_S^1\mathbb Z_p(O_{K,S}^{\times})_{\mathrm{tf}}.
\]
\end{lemma}

\begin{proof}
By the same argument as in
 \cite[Lemma 4.2]{El26}  or \cite[Lemma 6.9(2)]{GreitherPopescu} we get a description of the annihilator of the $p$-part of roots of unity as

{}

    \[
    \mathrm{Ann}_{\mathbb{Z}_p[G]}(\mu(K)\otimes_\mathbb{Z}\mathbb{Z}_p)=\langle \delta_{K/k,T}\;|\; T\neq \varnothing,\; T\cap S_p=\varnothing\rangle_{\mathbb{Z}_p[G]}.
    \]

This and the definition of rank-one Rubin--Stark elements (see \S \ref{Subsec:Rubin--Stark-global-fields}) yield

{}

    \[
    (\tilde{\lambda}_{K,S})^{-1}(\varepsilon_1^S\vartheta_{K/k,S}(0)X_{K,S})=\langle \eta^1_{K/k,S,T}\;|\; T\neq \varnothing,\; T\cap S_p=\varnothing\rangle_{\mathbb{Z}_p[G]},
    \]

{}

The result follows by Proposition \ref{Prop:RS-integrality-rank-one}.
\qed
\end{proof}

\begin{corollary}
\label{Cor:eS-index-well-defined}
Under the hypotheses of Lemma \ref{Lem:rho-generated-by-rank-one-RS-elements},
the classical index
\[
\left[
e_S^1\mathbb Z_p(O_{K,S}^{\times})_{\mathrm{tf}}
:
e_S^1\mathbb Z_p\varrho_{K/k,S}
\right]
\]
is well defined and belongs to \(\mathbb Z_{\geq 1}\).
\end{corollary}

\section{The rank-one component} \label{section:rank-one}
We now focus on the situation in which every non-trivial character has order of
vanishing one at \(s=0\). This is the range in which the generalized
Stickelberger module \(\varrho_{K/k,S}\) is governed by the first derivatives of
the Artin \(L\)-functions, and hence by rank-one Rubin--Stark elements.

More precisely, we assume that there exists a place
\[
\mathfrak v\in S
\]
which splits completely in \(K/k\), and that every place in
\[
S\setminus\{\mathfrak v\}
\]
does not decompose in \(K/k\) i.e., has decomposition group equal to
\(G\). We put
\[
S':=\{\mathfrak v\},
\qquad
m:=|S\setminus S'|.
\]
Then, for every non-trivial character \(\chi\in\widehat G\), one has
\begin{equation} \label{r-S-chi}
r_S(\chi)=1,
\end{equation}
whereas
\[
r_S(\chi_0)=|S|-1=m.
\]
Consequently,
\begin{equation} \label{e-S-1}
e_S^1=
\begin{cases}
1, & \text{if }m=1,\\[0.3em]
1-e_{\chi_0}, & \text{if }m\neq 1.
\end{cases}
\end{equation}
We shall treat the cases \(|S|=1\) and \(|S|=2\) separately (in \S\ref{subsection:S=1} and \S\ref{subsection:S=2}, respectively). Then we proceed to explain
why the same method becomes less effective when \(|S|\geq 3\) (see \S\ref{subsection:S>=3} and \S\ref{section:conclusion}).
If \(F\) is a finite extension of \(k\), we use the convention
\[
R_{F,S}
:=
\left|
\left(
X_{F,S}:
\lambda_{F,S}\bigl((O_{F,S}^{\times})_{\mathrm{tf}}\bigr)
\right)_{\mathbb Z}
\right|,
\]
with \(R_{F,S}=1\) when \(O_{F,S}^{\times}\) has rank zero. 
Then we get the following.
\begin{proposition}
\label{Prop:generalized-index-formula-rank-one}
The generalized index
\[
\left(
(O_{K,S}^{\times})_{\mathrm{tf}}:
\varrho_{K/k,S}
\right)_{\mathbb Z}
\]
is well defined, and one has
\[
\left(
(O_{K,S}^{\times})_{\mathrm{tf}}:
\varrho_{K/k,S}
\right)_{\mathbb Z}
=
\frac{h_{K,S}}{w_K}
\left(
\frac{h_{k,S}R_{k,S}}{w_k}
\right)^{|S\setminus S'|-1}
\left[
X_{K,S}:
\mathrm{Ann}_{\mathbb Z[G]}(\mu(K))X_{K,S}
\right]\mathbb Z.
\]
In particular:
\begin{itemize}
\item If \(|S|=1\), then
\[
\left(
(O_{K,S}^{\times})_{\mathrm{tf}}:
\varrho_{K/k,S}
\right)_{\mathbb Z}
=
\frac{h_{K,S}w_k}{h_{k,S}w_K}
\left[
X_{K,S}:
\mathrm{Ann}_{\mathbb Z[G]}(\mu(K))X_{K,S}
\right]\mathbb Z.
\]

\item If \(|S|=2\), then
\[
\left(
(O_{K,S}^{\times})_{\mathrm{tf}}:
\varrho_{K/k,S}
\right)_{\mathbb Z}
=
\frac{h_{K,S}}{w_K}
\left[
X_{K,S}:
\mathrm{Ann}_{\mathbb Z[G]}(\mu(K))X_{K,S}
\right]\mathbb Z.
\]
\end{itemize}
\end{proposition}

\begin{proof}
The argument follows the proof of \cite[Theorem~3.6(1)]{El26}, with the
\(S\)-modified objects kept throughout. By definition,
\[
\varrho_{K/k,S}
=
\widetilde{\lambda}_{K,S}^{-1}
\left(
\vartheta_{K/k,S}(0)X_{K,S}
\right),
\]
where
\[
\vartheta_{K/k,S}(0)
=
\mathrm{Ann}_{\mathbb Z[G]}(\mu(K))\Theta_{K/k,S}^{*}(0).
\]
Since \(\widetilde{\lambda}_{K,S}\) is an isomorphism, the generalized index is well defined and
\[
\left(
(O_{K,S}^{\times})_{\mathrm{tf}}:
\varrho_{K/k,S}
\right)_{\mathbb Z}
=
\left(
\lambda_{K,S}\bigl((O_{K,S}^{\times})_{\mathrm{tf}}\bigr):
\vartheta_{K/k,S}(0)X_{K,S}
\right)_{\mathbb Z}.
\]
Put
\(
A:=\mathrm{Ann}_{\mathbb Z[G]}(\mu(K)).
\)
By multiplicativity of generalized indices, we have
\[
\begin{aligned}
(
\lambda_{K,S}&\bigl((O_{K,S}^{\times})_{\mathrm{tf}}\bigr):
\vartheta_{K/k,S}(0)X_{K,S}
)_{\mathbb Z}\\
=&
\left(
\lambda_{K,S}\bigl((O_{K,S}^{\times})_{\mathrm{tf}}\bigr):
X_{K,S}
\right)_{\mathbb Z}        
\cdot
\left(
X_{K,S}:\Theta_{K/k,S}^{*}(0)X_{K,S}
\right)_{\mathbb Z}        
\cdot
\left(
\Theta_{K/k,S}^{*}(0)X_{K,S}:\Theta_{K/k,S}^{*}(0)AX_{K,S}
\right)_{\mathbb Z}.
\end{aligned}
\]
Since multiplication by $\Theta_{K/k,S}^{*}(0)$ induces an automorphism of the $\mathbb{C}$-vector space $\mathbb{C}X_{K,S}$, the last factor is equal to
\[
\left(
X_{K,S}:AX_{K,S}
\right)_{\mathbb Z}=[X_{K,S}:AX_{K,S}]\mathbb{Z}.
\]

The first factor is the inverse of the \(S\)-regulator of \(K\):
\[
\left(
\lambda_{K,S}\bigl((O_{K,S}^{\times})_{\mathrm{tf}}\bigr):X_{K,S}
\right)_{\mathbb Z}
=
\frac{1}{R_{K,S}}\mathbb Z.
\]

It remains to compute $\left(
X_{K,S}:\Theta_{K/k,S}^{*}(0)X_{K,S}
\right)_{\mathbb Z}$, which, by definition, is generated by
the determinant of multiplication by
\(\Theta_{K/k,S}^{*}(0)\) on \(\mathbb C X_{K,S}\). Since \(\mathfrak v\)
splits completely in \(K/k\), and since every place in \(S\setminus S'\) does
not decompose in \(K/k\), we have
\[
\dim_{\mathbb C}(e_{\chi_0}\mathbb C X_{K,S})=|S\setminus S'|
\]
and
\[
\dim_{\mathbb C}(e_\chi\mathbb C X_{K,S})=1
\qquad
\text{for every }\chi\neq \chi_0.
\]
Therefore multiplication by \(\Theta_{K/k,S}^{*}(0)\) on
\(\mathbb C X_{K,S}\) has determinant
\[
\zeta_{K,S}^{*}(0)
\left(\zeta_{k,S}^{*}(0)\right)^{|S\setminus S'|-1},
\]
up to sign. Hence
\[
\left(
X_{K,S}:
\Theta_{K/k,S}^{*}(0)X_{K,S}
\right)_{\mathbb Z}
=
\zeta_{K,S}^{*}(0)
\left(\zeta_{k,S}^{*}(0)\right)^{|S\setminus S'|-1}
\mathbb Z.
\]
By the \(S\)-class number formulae,
\[
\zeta_{K,S}^{*}(0)=\pm \frac{h_{K,S}R_{K,S}}{w_K}
\qquad
\text{and}
\qquad
\zeta_{k,S}^{*}(0)=\pm \frac{h_{k,S}R_{k,S}}{w_k}.
\]
Thus
\[
R_{K,S}^{-1}
\left(
X_{K,S}:
\Theta_{K/k,S}^{*}(0)X_{K,S}
\right)_{\mathbb Z}
=
\frac{h_{K,S}}{w_K}
\left(
\frac{h_{k,S}R_{k,S}}{w_k}
\right)^{|S\setminus S'|-1}
\mathbb Z.
\]
Combining this with the previous index decomposition gives
\[
\left(
(O_{K,S}^{\times})_{\mathrm{tf}}:
\varrho_{K/k,S}
\right)_{\mathbb Z}
=
\frac{h_{K,S}}{w_K}
\left(
\frac{h_{k,S}R_{k,S}}{w_k}
\right)^{|S\setminus S'|-1}
\left[
X_{K,S}:AX_{K,S}
\right]\mathbb Z.
\]
This proves the general formula.
If \(|S|=1\), then \(|S\setminus S'|=0\), and \(R_{k,S}=1\) since
\(X_{k,S}=0\). This gives the first special case. If \(|S|=2\), then
\(|S\setminus S'|=1\), and the second special case follows directly.\qed
\end{proof}

The only unspecified factor remaining in Proposition
\ref{Prop:generalized-index-formula-rank-one} is the index
\[
\left[
X_{K,S}:
\mathrm{Ann}_{\mathbb Z[G]}(\mu(K))X_{K,S}
\right].
\]
We compute it explicitly as follows.

\begin{proposition}
\label{Prop:XKS-annihilator-index-two-cases}
With the previous notations and assumptions, the following hold.
\begin{enumerate}
\item If \(|S|=1\), then
\[
\left[
X_{K,S}:
\mathrm{Ann}_{\mathbb Z[G]}(\mu(K))X_{K,S}
\right]
=
\frac{w_K}{w_k}
\left[
\mathrm{Ann}_{\mathbb Z[G]}(\mu(K))\cap\Delta G:
\mathrm{Ann}_{\mathbb Z[G]}(\mu(K))\Delta G
\right].
\]

\item If \(|S|\neq 1\), then
\[
\left[
X_{K,S}:
\mathrm{Ann}_{\mathbb Z[G]}(\mu(K))X_{K,S}
\right]
=
w_K\,w_k^{|S\setminus S'|-1}.
\]
\end{enumerate}
\end{proposition}

\begin{proof}
As before, put
\[
A:=\mathrm{Ann}_{\mathbb Z[G]}(\mu(K)).
\]
We shall use the following standard augmentation computation. The augmentation
map induces an exact sequence
\[
0
\longrightarrow
\frac{\Delta G}{A\cap \Delta G}
\longrightarrow
\frac{\mathbb Z[G]}{A}
\longrightarrow
\frac{\mathbb Z}{\mathrm{Aug}_G(A)}
\longrightarrow
0.
\]
Since \(\mu(K)\) is a cyclic \(\mathbb{Z}[G]\)-module, one has
\[
\mathbb Z[G]/A\simeq \mu(K),
\]
and therefore
\[
[\mathbb Z[G]:A]=w_K.
\]
Moreover, as in \cite[Proposition~4.4]{El26},
\begin{equation}\label{roots-of-units-augmentation}
\mathrm{Aug}_G(A)
=
\mathrm{Ann}_{\mathbb Z}(\mu(k))
=
w_k\mathbb Z.
\end{equation}

Taking cardinalities in the previous exact sequence,
we obtain
\[
\left[\Delta G:A\cap\Delta G\right]
=
\frac{
\left[\mathbb Z[G]:A\right]
}{
\left[\mathbb Z:\mathrm{Aug}_G(A)\right]
}.
\]
Since
%BEGIN DIFNOMARKUP
\begin{equation} 
\label{A-and-AugA-index}
[\mathbb Z[G]:A]=w_K
\qquad\text{and}\qquad
[\mathbb Z:\mathrm{Aug}_G(A)]=w_k,
\end{equation}
%END DIFNOMARKUP
it follows that
\begin{equation} \label{Delta-A-index}
\left[\Delta G:A\cap\Delta G\right]
=
\frac{w_K}{w_k}.
\end{equation}
We first assume that \(|S|=1\). Choose a place \(w_0\) of \(K\) above
\(\mathfrak v\). Since \(\mathfrak v\) splits completely in \(K/k\), we have
\[
Y_{K,S}\simeq \mathbb Z[G]w_0,
\qquad
X_{K,S}\simeq \Delta G.
\]
Hence
\[
AX_{K,S}\simeq A\Delta G,
\]
and therefore
\[
\left[
X_{K,S}:AX_{K,S}
\right]
=
\left[
\Delta G:A\Delta G
\right].
\]
Factoring through \(A\cap\Delta G\), we get
\[
\left[
\Delta G:A\Delta G
\right]
=
\left[
\Delta G:A\cap\Delta G
\right]
\left[
A\cap\Delta G:A\Delta G
\right].
\]
Using the computation above gives
\[
\left[
X_{K,S}:AX_{K,S}
\right]
=
\frac{w_K}{w_k}
\left[
A\cap\Delta G:A\Delta G
\right],
\]
which proves the assertion when \(|S|=1\).
We now assume that \(|S|\neq 1\). As before, we write
\[
m:=|S\setminus S'|.
\]
Then \(m\geq 1\). Choose a place \(w_0\) of \(K\) above \(\mathfrak v\), and
write
\[
S\setminus S'=\{v_1,\dots,v_m\}.
\]
For each \(i\), let \(w_i\) be the unique place of \(K\) above \(v_i\). Since
each \(v_i\) does not decompose in \(K/k\), the group \(G\) acts trivially on
\(w_i\). Thus
\[
Y_{K,S}
\simeq
\mathbb Z[G]w_0\oplus \bigoplus_{i=1}^m\mathbb Z w_i.
\]
Under this identification,
\[
X_{K,S}
=
\left\{
(a,n_1,\dots,n_m)\in \mathbb Z[G]\oplus\mathbb Z^m
\; ; \;
\mathrm{Aug}_G(a)+n_1+\cdots+n_m=0
\right\}.
\]
Projecting onto $\mathbb{Z}^m$, gives an exact sequence of \(\mathbb Z[G]\)-modules
\[
0
\longrightarrow
\Delta G
\longrightarrow
X_{K,S}
\longrightarrow
\mathbb Z^m
\longrightarrow
0,
\]
where \(G\) acts trivially on \(\mathbb Z^m\).
The image of \(AX_{K,S}\) in \(\mathbb Z^m\) is
\[
\mathrm{Aug}_G(A)\mathbb Z^m.
\]
We claim that
\begin{equation} \label{Ann-X-cap-Delta-G}
AX_{K,S}\cap \Delta G=A\cap \Delta G.
\end{equation}
Indeed, the inclusion from left to right is immediate from the above
description of \(X_{K,S}\). Conversely, let
\[
\alpha\in A\cap\Delta G.
\]
Since \(m\geq 1\), the element
\[
(1,-1,0,\dots,0)
\]
belongs to \(X_{K,S}\). Multiplying by \(\alpha\), we obtain
\[
\alpha(1,-1,0,\dots,0)
=
(\alpha,0,\dots,0),
\]
because \(\mathrm{Aug}_G(\alpha)=0\). Hence \((\alpha,0,\dots,0)\) belongs to
\(AX_{K,S}\). Under the identification
\[
\Delta G=\ker\left(X_{K,S}\longrightarrow \mathbb Z^m\right),
\]
the element \((\alpha,0,\dots,0)\) is precisely the image of \(\alpha\in\Delta G\).
Thus
\[
\alpha\in AX_{K,S}\cap\Delta G.
\]
This proves the claim.
Consequently, we have an exact sequence
\[
0
\longrightarrow
\frac{\Delta G}{A\cap\Delta G}
\longrightarrow
\frac{X_{K,S}}{AX_{K,S}}
\longrightarrow
\frac{\mathbb Z^m}{\mathrm{Aug}_G(A)\mathbb Z^m}
\longrightarrow
0.
\]
Taking orders gives
\[
\left[
X_{K,S}:AX_{K,S}
\right]
=
\left[
\Delta G:A\cap\Delta G
\right]
\left[
\mathbb Z:\mathrm{Aug}_G(A)
\right]^m.
\]
Using
\eqref{A-and-AugA-index} and \eqref{Delta-A-index},
we obtain
\[
\left[
X_{K,S}:AX_{K,S}
\right]
=
\frac{w_K}{w_k}w_k^m
=
w_Kw_k^{m-1}.
\]
This proves the assertion when \(|S|\ne1\), and hence the proposition.\qed
\end{proof}
\subsection{Some base-change computations}
In this subsection, we suppose as before that \(S\) is a finite set of places
containing one place \(\mathfrak v\) that is totally split in \(K/k\), and we set
\(S':=\{\mathfrak v\}\). Let \(\Sc\supseteq S\) be another finite set of
places of \(k\). We assume that \(|\Sc|\neq 2\) and that no place in
\(\Sc\setminus S'\) is totally split in \(K/k\).

We shall use the following elementary observation throughout this subsection.
Let \(v\in\Sc\setminus S'\), and let \(w\) be a place of \(K\) above
\(v\). Then
\[
e_{\Sc}^1 w=0.
\]
Indeed, let \(\chi\in\widehat G\) be a character such that \(e_\chi\) occurs in
\(e_{\Sc}^1\). Then \(r_{\Sc}(\chi)=1\). Since the place
\(\mathfrak v\) is totally split in \(K/k\), it already contributes one to
\(r_{\Sc}(\chi)\). Hence by \eqref{order-formula}, for every \(v\in\Sc\setminus S'\), one
must have
\[
\chi({D_v(K/k)})\neq \mathbf 1.
\]
By the standard vanishing property
\[
\chi({D_v(K/k)})\neq \mathbf 1
\implies
e_\chi w=0
\]
(see for example \cite[Proof of Lemma 2.6(i)]{Rubin}), it follows that
\(e_\chi w=0\) for every character \(\chi\) occurring in
\(e_{\Sc}^1\). Therefore
\begin{equation} \label{preliminary-observation}
e_{\Sc}^1w=0
\qquad
\text{for every }v\in\Sc\setminus S
\text{ and every }w\mid v.
\end{equation}
Notice also that the trivial character does not occur in \(e_{\Sc}^1\),
since
\[
r_{\Sc}(\chi_0)=|\Sc|-1\neq 1.
\]
\begin{lemma}\label{Lemma_4_8}
    If the places in $\Sc\setminus S'$ do not decompose in $K/k$, then
   \[
\left(
e_{\Sc}^1\mathbb Z_p\varrho_{K/k,S}
:
e_{\Sc}^1\mathbb Z_p\varrho_{K/k,\Sc}
\right)_{\mathbb Z_p}
=
\left(
\prod_{v\in\Sc\setminus S}[G:I_v]
\right)\mathbb Z_p.
\]
\end{lemma}
\begin{proof}
  By the preliminary observation \eqref{preliminary-observation}, we have 
\[
e_{\Sc}^1X_{K,\Sc}=e_{\Sc}^1X_{K,S}\qquad \text{and}\qquad e_{\Sc}^1\tilde{\lambda}_{K,\Sc}=e_{\Sc}^1\tilde{\lambda}_{K, S}.
\]
Next, recall that
(see e.g., \cite[(11)]{Solomon} or \cite[(7)]{DasKakde})

{}

\begin{equation*}
\Theta_{K/k,\Sc}^{*}(0)=\Big{(}\prod_{v\in \Sc\setminus S}(1 - \sigma_v^{-1}e_v)\Big{)}\cdot\Theta_{K/k,S}^{*}(0),
\end{equation*}

{}
where
\[
e_v:=\frac{1}{|I_v|}\sum_{\tau\in I_v}\tau.
\]
Put
\[
E_{\Sc/S}
:=
\prod_{v\in\Sc\setminus S}
(1-\sigma_v^{-1}e_v).
\]
Using the preceding identification of the regulator targets, the change from
\(e_{\Sc}^1\mathbb Z_p\varrho_{K/k,S}\) to
\(e_{\Sc}^1\mathbb Z_p\varrho_{K/k,\Sc}\) is given by multiplication by
\(E_{\Sc/S}\) on \(e_{\Sc}^1\mathbb Q_pX_{K,S}\). Hence
\[
\left(
e_{\Sc}^1\mathbb Z_p\varrho_{K/k,S}
:
e_{\Sc}^1\mathbb Z_p\varrho_{K/k,\Sc}
\right)_{\mathbb Z_p}
=
\mathrm{det}_{\mathbb Q_p}
\left(
E_{\Sc/S}\mid e_{\Sc}^1\mathbb Q_pX_{K,S}
\right)\mathbb Z_p.
\]

For a fixed \(v\in\Sc\setminus S\), since \(D_v=G\), the quotient
\(G/I_v\) is cyclic and generated by the image of \(\sigma_v\). On the
\(\chi\)-component, the factor \(1-\sigma_v^{-1}e_v\) acts as \(1\) if
\(\chi\) is non-trivial on \(I_v\), and as \(1-\chi(\sigma_v)^{-1}\) if
\(\chi\) is trivial on \(I_v\). Also, since $|\Sc|\neq 2$ and the places in $\Sc\setminus S'$ do not decompose in $K/k$, we have $e_{\Sc}^1=1-e_{\chi_0}$. Therefore
\[
\mathrm{det}_{\mathbb Q_p}
\left(
1-\sigma_v^{-1}e_v\mid e_{\Sc}^1\mathbb Q_pX_{K,S}
\right)
=
\prod_{\substack{\psi\in\widehat{G/I_v}\\ \psi\neq 1}}
(1-\psi(\sigma_v)^{-1})
=
[G:I_v].
\]
We explain the last equality: If \(N:=[G:I_v]\), then the values
\(\psi(\sigma_v)\), for \(\psi\in\widehat{G/I_v}\), are precisely the \(N\)-th roots
of unity. Hence
\[
\prod_{\substack{\psi\in\widehat{G/I_v}\\ \psi\neq 1}}
(1-\psi(\sigma_v)^{-1})
=
\prod_{j=1}^{N-1}(1-\zeta^j),
\]
where \(\zeta\) is a primitive \(N\)-th root of unity. Since
\[
1+X+\cdots+X^{N-1}
=
\prod_{j=1}^{N-1}(X-\zeta^j),
\]
evaluating at \(X=1\) gives
\[
\prod_{j=1}^{N-1}(1-\zeta^j)=N.
\]
Thus
\[
\left(
e_{\Sc}^1\mathbb Z_p\varrho_{K/k,S}
:
e_{\Sc}^1\mathbb Z_p\varrho_{K/k,\Sc}
\right)_{\mathbb Z_p}
=
\left(
\prod_{v\in\Sc\setminus S}[G:I_v]
\right)\mathbb Z_p.
\]
\qed
\end{proof}

\begin{lemma}
\label{Lem:unit-index-bound}
The index
\[
\left[
e_{\Sc}^1\mathbb Z_p(O_{K,\Sc}^{\times})_{\mathrm{tf}}
:
e_{\Sc}^1\mathbb Z_p(O_{K,S}^{\times})_{\mathrm{tf}}
\right]
\]
is well defined. Moreover,
\[
\left[
e_{\Sc}^1\mathbb Z_p(O_{K,\Sc}^{\times})_{\mathrm{tf}}
:
e_{\Sc}^1\mathbb Z_p(O_{K,S}^{\times})_{\mathrm{tf}}
\right]
\leq
p^{v_p(|G|)\cdot\left|\Sc_K\setminus S_K\right|}.
\]
In particular, if every place in \(\Sc\setminus S\) does not decompose in
\(K/k\), then
\[
\left[
e_{\Sc}^1\mathbb Z_p(O_{K,\Sc}^{\times})_{\mathrm{tf}}
:
e_{\Sc}^1\mathbb Z_p(O_{K,S}^{\times})_{\mathrm{tf}}
\right]
\leq
p^{v_p(|G|)\cdot|\Sc\setminus S|}.
\]
\end{lemma}

\begin{proof}
We add the places of \(\Sc\setminus S\) one at a time. Let
\[
S\subseteq S_1\subseteq \Sc,
\qquad
S_2=S_1\cup\{v\},
\]
and let \(w_1,\dots,w_t\) be the places of \(K\) above \(v\).

We use the standard exact sequence, see e.g.~\cite[Chapter~I, Proposition~11.6]{Neukirch},
\begin{equation} \label{S1-S2-units}
0
\longrightarrow
O_{K,S_1}^{\times}
\longrightarrow
O_{K,S_2}^{\times}
\longrightarrow
\bigoplus_{w\mid v}\mathbb Z w
\longrightarrow
\mathrm{Cl}_{K,S_1}
\longrightarrow
\mathrm{Cl}_{K,S_2}
\longrightarrow
0,
\end{equation}
where the third map is induced by taking the divisor outside \(S_1\). Since
\(\mu(K)\subseteq O_{K,S_1}^{\times}\), this induces an exact sequence
\[
0
\longrightarrow
(O_{K,S_1}^{\times})_{\mathrm{tf}}
\longrightarrow
(O_{K,S_2}^{\times})_{\mathrm{tf}}
\longrightarrow
\bigoplus_{w\mid v}\mathbb Z w
\longrightarrow
\mathrm{Cl}_{K,S_1}
\longrightarrow
\mathrm{Cl}_{K,S_2}.
\]

Upon tensoring with \(\mathbb Q_p\), the class groups vanish. Hence
\[
0
\longrightarrow
\mathbb Q_p(O_{K,S_1}^{\times})_{\mathrm{tf}}
\longrightarrow
\mathbb Q_p(O_{K,S_2}^{\times})_{\mathrm{tf}}
\longrightarrow
\bigoplus_{w\mid v}\mathbb Q_p w
\longrightarrow
0
\]
is exact. By the preliminary observation \eqref{preliminary-observation}, we have
\[
e_{\Sc}^1w=0
\qquad
\text{for every }w\mid v.
\]
and we obtain
\[
e_{\Sc}^1\mathbb Q_p(O_{K,S_2}^{\times})_{\mathrm{tf}}
=
e_{\Sc}^1\mathbb Q_p(O_{K,S_1}^{\times})_{\mathrm{tf}}.
\]
Therefore
\(
{
e_{\Sc}^1\mathbb Z_p(O_{K,S_2}^{\times})_{\mathrm{tf}}
}/{
e_{\Sc}^1\mathbb Z_p(O_{K,S_1}^{\times})_{\mathrm{tf}}
}
\)
is finite.
We now bound its order. Tensoring the exact sequence \eqref{S1-S2-units} with
\(\mathbb Z_p\), we get an injection
\[
\frac{
\mathbb Z_p(O_{K,S_2}^{\times})_{\mathrm{tf}}
}{
\mathbb Z_p(O_{K,S_1}^{\times})_{\mathrm{tf}}
}
\hookrightarrow
\bigoplus_{w\mid v}\mathbb Z_p w
=
\bigoplus_{i=1}^{t}\mathbb Z_p w_i.
\]
Since \(\mathbb Z_p\) is a principal ideal domain, every submodule of a free
\(\mathbb Z_p\)-module is free. Hence
\[
\frac{
\mathbb Z_p(O_{K,S_2}^{\times})_{\mathrm{tf}}
}{
\mathbb Z_p(O_{K,S_1}^{\times})_{\mathrm{tf}}
}
\]
can be generated by \(t\) elements. Moreover, the natural map
\[
\frac{
\mathbb Z_p(O_{K,S_2}^{\times})_{\mathrm{tf}}
}{
\mathbb Z_p(O_{K,S_1}^{\times})_{\mathrm{tf}}
}
\longrightarrow
\frac{
e_{\Sc}^1\mathbb Z_p(O_{K,S_2}^{\times})_{\mathrm{tf}}
}{
e_{\Sc}^1\mathbb Z_p(O_{K,S_1}^{\times})_{\mathrm{tf}}
},
\qquad
u\longmapsto e_{\Sc}^1u,
\]
is surjective. Therefore the latter quotient can also be generated by \(t\)
elements.
It remains to bound the exponent. Since
\[
|G|e_{\Sc}^1\in \mathbb Z_p[G],
\]
we have
\[
|G|e_{\Sc}^1u
\in
\mathbb Z_p(O_{K,S_2}^{\times})_{\mathrm{tf}}
\]
for every \(u\in \mathbb Z_p(O_{K,S_2}^{\times})_{\mathrm{tf}}\). Its image in
\(
\bigoplus_{w\mid v}\mathbb Z_p w
\)
is obtained by applying \(|G|e_{\Sc}^1\) to the image of \(u\), and is
therefore zero because \(e_{\Sc}^1w=0\) for every \(w\mid v\). Hence
\[
|G|e_{\Sc}^1u
\in
\mathbb Z_p(O_{K,S_1}^{\times})_{\mathrm{tf}}.
\]
Since this element lies in the \(e_{\Sc}^1\)-component, we get
\[
|G|e_{\Sc}^1u
\in
e_{\Sc}^1\mathbb Z_p(O_{K,S_1}^{\times})_{\mathrm{tf}}.
\]
Thus
\[
|G|\cdot e_{\Sc}^1\mathbb Z_p(O_{K,S_2}^{\times})_{\mathrm{tf}}
\subseteq
e_{\Sc}^1\mathbb Z_p(O_{K,S_1}^{\times})_{\mathrm{tf}}.
\]
Therefore the finite quotient
\[
\frac{
e_{\Sc}^1\mathbb Z_p(O_{K,S_2}^{\times})_{\mathrm{tf}}
}{
e_{\Sc}^1\mathbb Z_p(O_{K,S_1}^{\times})_{\mathrm{tf}}
}
\]
is killed by \(|G|\). Since it can be generated by \(t\) elements, its order is
bounded by
\(
p^{v_p(|G|)t}
\).
Iterating over all places \(v\in\Sc\setminus S\), and using
multiplicativity of indices, we obtain
\[
\left[
e_{\Sc}^1\mathbb Z_p(O_{K,\Sc}^{\times})_{\mathrm{tf}}
:
e_{\Sc}^1\mathbb Z_p(O_{K,S}^{\times})_{\mathrm{tf}}
\right]
\leq
p^{v_p(|G|)\|\Sc_K\setminus S_K|}.
\]
If every place in \(\Sc\setminus S\) does not decompose in \(K/k\), then
each such place has a unique place above it in \(K\). Hence
\[
|\Sc_K\setminus S_K|=|\Sc\setminus S|,
\]
and the final assertion follows.\qed
\end{proof}

\subsection{The case \( |S|=1 \)} \label{subsection:S=1}
We now specialize to the case where
\[
S=S'=\{\mathfrak v\},
\]
with \(\mathfrak v\) totally split in \(K/k\). By Proposition
\ref{Prop:XKS-annihilator-index-two-cases}, the remaining correction term in
the generalized index is
\[
\left[
\mathrm{Ann}_{\mathbb Z[G]}(\mu(K))\cap\Delta G:
\mathrm{Ann}_{\mathbb Z[G]}(\mu(K))\Delta G
\right].
\]
In \cite[\S 4.4]{El26}, this factor was shown to have trivial \(p\)-primary
part whenever \(p\nmid |G|\). We now give a cohomological description of its
\(p\)-primary part for every rational prime \(p\), including the primes dividing
\(|G|\).
\subsubsection{The annihilator-augmentation correction factor}
Let \(p\) be a rational prime, and put
\(
\mu(K)(p):=\mu(K)\otimes_{\mathbb Z}\mathbb Z_p.
\).
\begin{proposition}
\label{Prop:correction-factor-homology}
Let \(p\) be a rational prime. Then
\[
\frac{
\mathrm{Ann}_{\mathbb Z_p[G]}(\mu(K)(p))\cap \Delta_pG
}{
\mathrm{Ann}_{\mathbb Z_p[G]}(\mu(K)(p))\Delta_pG
}
\simeq
H_1(G,\mu(K)(p)).
\]
In particular, if \(|\mu(K)(p)|=p^n\), then
\[
\left[
\mathrm{Ann}_{\mathbb Z_p[G]}(\mu(K)(p))\cap \Delta_pG:
\mathrm{Ann}_{\mathbb Z_p[G]}(\mu(K)(p))\Delta_pG
\right]
=
\big| H^1\left(G,\mathbb Z/p^n\mathbb Z(-1)\right) \big|.
\]
\end{proposition}

\begin{proof}
If \(\mu(K)(p)=0\), then
\[
\mathrm{Ann}_{\mathbb Z_p[G]}(\mu(K)(p))=\mathbb Z_p[G],
\]
and both sides of the first claimed isomorphism are zero. We may therefore
assume that \(\mu(K)(p)\neq 0\).
Choose a generator \(\zeta\) of the cyclic \(\mathbb Z_p[G]\)-module
\(\mu(K)(p)\). Then
\[
\mathbb Z_p[G]\longrightarrow \mu(K)(p),
\qquad
x\longmapsto x\zeta
\]
is surjective, with kernel
\[
\mathrm{Ann}_{\mathbb Z_p[G]}(\mu(K)(p)).
\]
Hence we have an exact sequence
\[
0
\longrightarrow
\mathrm{Ann}_{\mathbb Z_p[G]}(\mu(K)(p))
\longrightarrow
\mathbb Z_p[G]
\longrightarrow
\mu(K)(p)
\longrightarrow
0.
\]
Tensoring over \(\mathbb Z_p[G]\) with the trivial module
\(
\mathbb Z_p\simeq \mathbb Z_p[G]/\Delta_pG
\),
gives an exact sequence
\[
0
\longrightarrow
\mathrm{Tor}_1^{\mathbb Z_p[G]}(\mu(K)(p),\mathbb Z_p)
\longrightarrow
\frac{
\mathrm{Ann}_{\mathbb Z_p[G]}(\mu(K)(p))
}{
\mathrm{Ann}_{\mathbb Z_p[G]}(\mu(K)(p))\Delta_pG
}
\longrightarrow
\mathbb Z_p[G]/\Delta_pG.
\]
The last map is induced by the inclusion
\[
\mathrm{Ann}_{\mathbb Z_p[G]}(\mu(K)(p))
\hookrightarrow
\mathbb Z_p[G],
\]
and its kernel is
\[
\frac{
\mathrm{Ann}_{\mathbb Z_p[G]}(\mu(K)(p))\cap \Delta_pG
}{
\mathrm{Ann}_{\mathbb Z_p[G]}(\mu(K)(p))\Delta_pG
}.
\]
Thus
\[
\frac{
\mathrm{Ann}_{\mathbb Z_p[G]}(\mu(K)(p))\cap \Delta_pG
}{
\mathrm{Ann}_{\mathbb Z_p[G]}(\mu(K)(p))\Delta_pG
}
\simeq
\mathrm{Tor}_1^{\mathbb Z_p[G]}(\mu(K)(p),\mathbb Z_p).
\]
Since
\[
\mathrm{Tor}_1^{\mathbb Z_p[G]}(\mu(K)(p),\mathbb Z_p)
\simeq
H_1(G,\mu(K)(p)),
\]
the first assertion follows.
Pontryagin duality gives
\[
H_1(G,\mu(K)(p))^\vee
\simeq
H^1(G,\mu(K)(p)^\vee).
\]
Finally, if \(|\mu(K)(p)|=p^n\), then
\[
\mu(K)(p)\simeq \mathbb Z/p^n\mathbb Z(1),
\qquad
\mu(K)(p)^\vee\simeq \mathbb Z/p^n\mathbb Z(-1).
\]
Therefore
\[
\left[
\mathrm{Ann}_{\mathbb Z_p[G]}(\mu(K)(p))\cap \Delta_pG:
\mathrm{Ann}_{\mathbb Z_p[G]}(\mu(K)(p))\Delta_pG
\right]
=
\big| H_1(G,\mu(K)(p)) \big|
=
\big| H^1\left(G,\mathbb Z/p^n\mathbb Z(-1)\right) \big|.
\]\qed
\end{proof}
\begin{corollary}
\label{Cor:correction-factor-trivial}
With the notation above, one has
\[
\left[
\mathrm{Ann}_{\mathbb Z_p[G]}(\mu(K)(p))\cap \Delta_pG:
\mathrm{Ann}_{\mathbb Z_p[G]}(\mu(K)(p))\Delta_pG
\right]=1
\]
in each of the following cases:
\begin{enumerate}
    \item \(p\nmid |G|\);
    \item \(\mu(k)(p)=0\), equivalently \(p\nmid w_k\).
\end{enumerate}
\end{corollary}

\begin{proof}
By Proposition \ref{Prop:correction-factor-homology}, it is enough to show that
\[
H_1(G,\mu(K)(p))=0.
\]

If \(p\nmid |G|\), then \(H_1(G,\mu(K)(p))=0\), since group homology is killed by
\(|G|\), while multiplication by \(|G|\) is an automorphism of the \(p\)-primary
group \(\mu(K)(p)\). Observe that this recovers the result in \cite[\S 4.4]{El26}, but in a completely new way.

Assume now that \(\mu(k)(p)=0\). If \(\mu(K)(p)=0\), there is nothing to prove.
Suppose therefore that \(\mu(K)(p)\neq 0\). Since \(K/k\) is Galois, one has
\[
\mu(K)(p)^G=\mu(k)(p)=0.
\]
We claim that there exists \(\sigma\in G\) such that \(\sigma-1\) acts bijectively
on \(\mu(K)(p)\). Indeed, since \(\mu(K)(p)\) is cyclic, we may write
\[
\mu(K)(p)\simeq \mathbb Z/p^n\mathbb Z(1)
\]
for some \(n\geq 1\). Thus the action of \(G\) on \(\mu(K)(p)\) is given by a
character
\[
\kappa:G\longrightarrow (\mathbb Z/p^n\mathbb Z)^\times.
\]
If \(\kappa(\sigma)-1\) were non-invertible for every \(\sigma\in G\), then
\[
\kappa(\sigma)\equiv 1 \pmod p
\qquad
\text{for every }\sigma\in G.
\]
Hence \(G\) would act trivially on the non-zero subgroup
\[
\mu_p(K):=\{x\in \mu(K)\mid x^p=1\}\subseteq \mu(K)(p).
\]
This would imply
\[
\mu(K)(p)^G\neq 0,
\]
contradicting
\[
\mu(K)(p)^G=\mu(k)(p)=0.
\]
Therefore there exists some \(\sigma\in G\) such that
\[
\kappa(\sigma)-1\in(\mathbb Z/p^n\mathbb Z)^\times,
\]
or equivalently such that \(\sigma-1\) acts bijectively on \(\mu(K)(p)\).

Let
\[
H:=\langle\sigma\rangle.
\]
Since \(G\) is abelian, \(H\) is normal in \(G\). We first show that
\[
H_i(H,\mu(K)(p))=0
\qquad
\text{for all } i\geq 0.
\]
Indeed, the homology of the cyclic group \(H\) is computed by the standard
periodic complex whose differentials are \(\sigma-1\) and
\[
N_H:=1+\sigma+\cdots+\sigma^{|H|-1}.
\]
Since \(\sigma-1\) is an automorphism of \(\mu(K)(p)\), this complex is acyclic.
Now the Hochschild-Serre spectral sequence associated with
\[
1\longrightarrow H\longrightarrow G\longrightarrow G/H\longrightarrow 1
\]
has
\[
E^2_{a,b}=H_a(G/H,H_b(H,\mu(K)(p)))
\Longrightarrow H_{a+b}(G,\mu(K)(p)).
\]
Since \(H_b(H,\mu(K)(p))=0\) for all \(b\geq 0\), all \(E^2\)-terms vanish. Hence
\[
H_i(G,\mu(K)(p))=0
\qquad
\text{for all } i\geq 0.
\]
In particular,
\[
H_1(G,\mu(K)(p))=0.
\]
The result follows from Proposition \ref{Prop:correction-factor-homology}.\qed
\end{proof}
\begin{corollary}
\label{Cor:correction-factor-bound}
Assume that
\[
|\mu(K)(p)|=p^n.
\]
Let \(d_p(G)\) denote the minimal number of generators of the \(p\)-Sylow
subgroup of \(G\). Then
\[
\left[
\mathrm{Ann}_{\mathbb Z_p[G]}(\mu(K)(p))\cap\Delta_pG:
\mathrm{Ann}_{\mathbb Z_p[G]}(\mu(K)(p))\Delta_pG
\right]
\leq
p^{\min(n,v_p(|G|))\,d_p(G)}.
\]
\end{corollary}

\begin{proof}
By Proposition \ref{Prop:correction-factor-homology}, we have
\[
\left[
\mathrm{Ann}_{\mathbb Z_p[G]}(\mu(K)(p))\cap\Delta_pG:
\mathrm{Ann}_{\mathbb Z_p[G]}(\mu(K)(p))\Delta_pG
\right]
=
\big| H_1(G,\mu(K)(p)) \big|.
\]
Thus it is enough to bound the order of \(H_1(G,\mu(K)(p))\).

Let \(P\) be the \(p\)-Sylow subgroup of \(G\). Since \(\mu(K)(p)\) is
\(p\)-primary, the group \(H_1(G,\mu(K)(p))\) is \(p\)-primary. Moreover, the
restriction-corestriction relation gives
\[
\mathrm{cor}\circ \mathrm{res}=[G:P]
\]
on \(H_1(G,\mu(K)(p))\). Since \([G:P]\) is prime to \(p\), multiplication by
\([G:P]\) is an automorphism of \(H_1(G,\mu(K)(p))\). Hence the restriction map
is bijective, and therefore
\[
\big| H_1(G,\mu(K)(p)) \big|
=
\big| H_1(P,\mu(K)(p)) \big|.
\]

Set
\[
d:=d_p(G),
\]
and choose generators
\[
\tau_1,\ldots,\tau_d
\]
of \(P\). Set
\[
M:=\mu(K)(p).
\]
Since \(P\) is generated by \(\tau_1,\dots,\tau_d\), the augmentation ideal
\(\Delta P\subseteq \mathbb Z[P]\) is generated, as a right
\(\mathbb Z[P]\)-module, by
\[
\tau_1-1,\dots,\tau_d-1.
\]
Thus the beginning of a free resolution of the trivial
\(\mathbb Z[P]\)-module \(\mathbb Z\) may be chosen in the form
\[
F_2
\longrightarrow
\mathbb Z[P]^d
\longrightarrow
\mathbb Z[P]
\longrightarrow
\mathbb Z
\longrightarrow 0,
\]
where the map \(\mathbb Z[P]^d\to \mathbb Z[P]\) sends the \(i\)-th basis
vector to \(\tau_i-1\) and \(\mathbb Z[P]\to \mathbb Z\) is the augmentation map. After tensoring over \(\mathbb Z[P]\) with \(M\), the
degree-one term becomes \(M^d\). Hence
\[
H_1(P,M)
\]
is a subquotient of \(M^d\).
Since \(M\simeq \mathbb Z/p^n\mathbb Z\), the group \(M^d\) is a finite abelian
\(p\)-group generated by \(d\) elements and killed by \(p^n\). Any subquotient
of \(M^d\) is therefore generated by at most \(d\) elements and is killed by
\(p^n\). Since \(H_1(P,M)\) is a subquotient of \(M^d\), it follows that
\(H_1(P,M)\) is generated by at most \(d\) elements, each of order dividing
\(p^n\).

On the other hand, group homology in positive degree is killed by the order of
the group. Hence \(H_1(P,\mu(K)(p))\) is killed by
\[
|P|=p^{v_p(|G|)}.
\]
Consequently each of its generators has order dividing
\[
p^{\min(n,v_p(|G|))}.
\]
It follows that
\[
\big| H_1(P,\mu(K)(p)) \big|
\leq
p^{\min(n,v_p(|G|))\,d}.
\]
Therefore
\[
\big| H_1(G,\mu(K)(p)) \big|
\leq
p^{\min(n,v_p(|G|))\,d_p(G)}..
\]
The result follows.\qed
\end{proof}

\subsubsection{A \(p\)-primary refinement of the rank-one index formula}
Combining the previous results yields the following theorem.
\begin{theorem}
\label{Theo:p-primary-rank-one-index-formula}
Let \(p\) be a rational prime, and write
\[
|\mu(K)(p)|=p^n.
\]
Then the following assertions hold.
\begin{enumerate}
    \item One has
    \[
    \left(
    (O_{K,S}^{\times})_{\mathrm{tf}}:
    \varrho_{K/k,S}
    \right)_{\mathbb{Z}}
    \otimes_{\mathbb{Z}}\mathbb{Z}_p
    =
    \frac{h_{K,S}}{h_{k,S}}\,
    \big| H^1\left(G,\mathbb{Z}/p^n\mathbb{Z}(-1)\right)\! \big| \,
    \mathbb{Z}_p.
    \]

    \item If either
    \[
    p\nmid |G|
    \qquad
    \text{or}\qquad
    \mu(k)(p)=0,
    \]
     then
    \[
    \left(
    (O_{K,S}^{\times})_{\mathrm{tf}}:
    \varrho_{K/k,S}
    \right)_{\mathbb{Z}}
    \otimes_{\mathbb{Z}}\mathbb{Z}_p
    =
    \frac{h_{K,S}}{h_{k,S}}\mathbb{Z}_p.
    \]

    \item Let \(d_p(G)\) denote the minimal number of generators of the
    \(p\)-Sylow subgroup of \(G\). Then
    \[
    v_p\left(
    \left(
    (O_{K,S}^{\times})_{\mathrm{tf}}:
    \varrho_{K/k,S}
    \right)_{\mathbb{Z}}
    \right)
    \leq
    v_p\left(\frac{h_{K,S}}{h_{k,S}}\right)
    +
    \min(n,v_p(|G|))\,d_p(G).
    \]
\end{enumerate}
\end{theorem}

\begin{proof}
The first claim follows from
Proposition \ref{Prop:generalized-index-formula-rank-one} combined with Proposition \ref{Prop:XKS-annihilator-index-two-cases}($1$) and Proposition \ref{Prop:correction-factor-homology}. The second claim is obtained from the first by applying Corollary \ref{Cor:correction-factor-trivial}, while the last claim is obtained by Corollary \ref{Cor:correction-factor-bound}. \qed
\end{proof}

 \subsubsection{Integral rank-one Stark elements and class number divisibilities}

We begin with a simple comparison lemma for generalized indices. It will be used
to separate the contribution of an idempotent from its complementary component.

Let \(\mathcal{O}\) be a discrete valuation ring with uniformizer \(\pi\) and
field of fractions \(E\). Let \(G\) be a finite abelian group, and let \(V\) be a
finite-dimensional \(E[G]\)-module. Let \(M\) and \(N\) be \(\mathcal{O}[G]\)-modules
which are full \(\mathcal{O}\)-lattices in \(V\). Let
\[
e\in E[G]
\]
be an idempotent, and set
\[
\widehat{M}:=eM\oplus(1-e)M,
\qquad
\widehat{N}:=eN\oplus(1-e)N.
\]
\begin{lemma}
\label{Lem:idempotent-index-comparison}
We have
\[
(eM:eN)_{\mathcal{O}}
\bigl((1-e)M:(1-e)N\bigr)_{\mathcal{O}}
=
\pi^{
\ell_{\mathcal{O}}\left(eM/(eM\cap M)\right)
-
\ell_{\mathcal{O}}\left(eN/(eN\cap N)\right)
}
(M:N)_{\mathcal{O}},
\]
where \(\ell_{\mathcal{O}}\) denotes length over \(\mathcal{O}\).
Moreover, if \(t\geq 1\) is an integer such that
\[
te\in \mathcal{O}[G],
\]
then
\[
0\leq
\ell_{\mathcal{O}}\left(eM/(eM\cap M)\right)
\leq
v_{\pi}(t)\dim_E(eV).
\]
\end{lemma}

\begin{proof}
Since
\[
\widehat{M}=eM\oplus(1-e)M
\qquad\text{and}\qquad
\widehat{N}=eN\oplus(1-e)N,
\]
the generalized index from \(\widehat{M}\) to \(\widehat{N}\) decomposes as
\[
(\widehat{M}:\widehat{N})_{\mathcal{O}}
=
(eM:eN)_{\mathcal{O}}
\bigl((1-e)M:(1-e)N\bigr)_{\mathcal{O}}.
\]
On the other hand, by multiplicativity of generalized indices, see e.g.~\cite[\S3, Proposition~1(i)]{Froehlich},
\[
(\widehat{M}:\widehat{N})_{\mathcal{O}}
=
(\widehat{M}:M)_{\mathcal{O}}
(M:N)_{\mathcal{O}}
(N:\widehat{N})_{\mathcal{O}}.
\]
Since \(M\subseteq \widehat{M}\) and \(N\subseteq \widehat{N}\), we have
\[
(\widehat{M}:M)_{\mathcal{O}}
=
\pi^{\ell_{\mathcal{O}}(\widehat{M}/M)}\mathcal{O},
\qquad
(N:\widehat{N})_{\mathcal{O}}
=
\pi^{-\ell_{\mathcal{O}}(\widehat{N}/N)}\mathcal{O}.
\]
Moreover, the inclusion \(eM\hookrightarrow \widehat M\) induces an isomorphism
\[
eM/(eM\cap M)\xrightarrow{\;\sim\;} \widehat M/M,
\]
and similarly
\[
eN/(eN\cap N)\xrightarrow{\;\sim\;} \widehat N/N.
\]
Combining these identities proves the first assertion.
Assume now that \(te\in \mathcal{O}[G]\). Since \(M\) is an \(\mathcal{O}[G]\)-module,
multiplication by \(te\) preserves \(M\). Hence
\[
t\,eM\subseteq eM\cap M.
\]
Therefore \(eM/(eM\cap M)\) is a quotient of \(eM/t\,eM\). Since \(eM\) is a full
\(\mathcal{O}\)-lattice in \(eV\), we get
\[
\ell_{\mathcal{O}}\left(eM/(eM\cap M)\right)
\leq
\ell_{\mathcal{O}}(eM/t\,eM)
=
v_{\pi}(t)\dim_E(eV).
\]
The lower bound is immediate.\qed
\end{proof}

Let us fix some additional notation for the next result:
\begin{itemize}
\item Set $|\mu(K)(p)|=p^n$.
\item Like above, let $\Sc\supseteq S$ be any finite set of places of $k$ such that $|\Sc|\neq 2$.
   \end{itemize}
\begin{theorem}
\label{Thm:Ben-1-5-general-form}
Assume that every place in \(\Sc\setminus S\) does not decompose in
\(K/k\). Then
\begin{align*}
\Big(
e_{\Sc}^1&\mathbb Z_p(O_{K,\Sc}^{\times})_{\mathrm{tf}}
:
e_{\Sc}^1\mathbb Z_p\varrho_{K/k,\Sc}
\Big)_{\mathbb Z_p}\\
&=
\frac{h_{K,S}}{h_{k,S}}
\big| H^1(G,\mathbb Z/p^n\mathbb Z(-1)) \big|
\left[
e_{\Sc}^1\mathbb Z_p(O_{K,\Sc}^{\times})_{\mathrm{tf}}
:
e_{\Sc}^1\mathbb Z_p(O_{K,S}^{\times})_{\mathrm{tf}}
\right]\left(
e_{\Sc}^1\mathbb Z_p\varrho_{K/k,S}
:
e_{\Sc}^1\mathbb Z_p\varrho_{K/k,\Sc}
\right)_{\mathbb Z_p}.
\end{align*}
\end{theorem}
\begin{proof}
We apply Lemma \ref{Lem:idempotent-index-comparison} with
\[
\mathcal O=\mathbb Z_p,\qquad E=\mathbb Q_p,\qquad
V=\mathbb Q_p(O_{K,S}^{\times})_{\mathrm{tf}},
\]
to the two \(\mathbb Z_p\)-lattices
\[
M=\mathbb Z_p(O_{K,S}^{\times})_{\mathrm{tf}},
\qquad
N=\mathbb Z_p\varrho_{K/k,S},
\]
and to the idempotent \(e=e_{\Sc}^1\).
Together with Theorem \ref{Theo:p-primary-rank-one-index-formula}\((1)\), this
gives
\begin{equation} \label{e-units-rho-index}
\begin{gathered}
\left(
e_{\Sc}^1\mathbb Z_p(O_{K,S}^{\times})_{\mathrm{tf}}
:
e_{\Sc}^1\mathbb Z_p\varrho_{K/k,S}
\right)_{\mathbb Z_p}
\cdot
\left(
(1-e_{\Sc}^1)\mathbb Z_p(O_{K,S}^{\times})_{\mathrm{tf}}
:
(1-e_{\Sc}^1)\mathbb Z_p\varrho_{K/k,S}
\right)_{\mathbb Z_p}
\\
=
p^{
\ell_{\mathbb Z_p}\left(
e_{\Sc}^1M/(e_{\Sc}^1M\cap M)
\right)
-
\ell_{\mathbb Z_p}\left(
e_{\Sc}^1N/(e_{\Sc}^1N\cap N)
\right)
}
\frac{h_{K,S}}{h_{k,S}}
\big| H^1(G,\mathbb Z/p^n\mathbb Z(-1)) \big| \, \mathbb Z_p.
\end{gathered}
\end{equation}

We now simplify this formula. Since we are in the case \(|S|=1\), with
\(S=\{\mathfrak v\}\) and \(\mathfrak v\) totally split in \(K/k\), we have
\[
X_{K,S}\simeq \Delta G.
\]
Moreover, since every place in \(\Sc\setminus S\) does not decompose in
\(K/k\), we have
\[
e_{\Sc}^1=1-e_{\chi_0}.
\]
Thus \(e_{\Sc}^1\) acts as the identity on
\[
\mathbb Q_pX_{K,S}\simeq \mathbb Q_p\Delta G
\]
(since $(1-e_{\chi_0})(1-\sigma)=(1-\sigma)$ for any $\sigma\in G$).
Via the regulator isomorphism, it follows that \(e_{\Sc}^1\) acts also as the
identity on
\[
\mathbb Q_p(O_{K,S}^{\times})_{\mathrm{tf}}.
\]
The same holds for \(\mathbb Q_p\varrho_{K/k,S}\), by the definition of
\(\varrho_{K/k,S}\).
Consequently,
\[
(1-e_{\Sc}^1)\mathbb Q_p(O_{K,S}^{\times})_{\mathrm{tf}}=0
\qquad
\text{and}
\qquad
(1-e_{\Sc}^1)\mathbb Q_p\varrho_{K/k,S}=0.
\]
Therefore
\[
\left(
(1-e_{\Sc}^1)\mathbb Z_p(O_{K,S}^{\times})_{\mathrm{tf}}
:
(1-e_{\Sc}^1)\mathbb Z_p\varrho_{K/k,S}
\right)_{\mathbb Z_p}
=
\mathbb Z_p.
\]

Furthermore, since \(e_{\Sc}^1\) acts as the identity on the ambient
space \(V\), we have
\[
e_{\Sc}^1M=M,
\qquad
e_{\Sc}^1N=N.
\]
Hence
\[
\ell_{\mathbb Z_p}\left(
e_{\Sc}^1M/(e_{\Sc}^1M\cap M)
\right)=0
\]
and
\[
\ell_{\mathbb Z_p}\left(
e_{\Sc}^1N/(e_{\Sc}^1N\cap N)
\right)=0.
\]
The formula \eqref{e-units-rho-index} therefore reduces to
\[
\left(
e_{\Sc}^1\mathbb Z_p(O_{K,S}^{\times})_{\mathrm{tf}}
:
e_{\Sc}^1\mathbb Z_p\varrho_{K/k,S}
\right)_{\mathbb Z_p}
=
\frac{h_{K,S}}{h_{k,S}}
\big| H^1(G,\mathbb Z/p^n\mathbb Z(-1)) \big| \,\mathbb Z_p.
\]
Moreover, by Lemma \ref{Lem:unit-index-bound}, the index
\[
\left[
e_{\Sc}^1\mathbb Z_p(O_{K,\Sc}^{\times})_{\mathrm{tf}}
:
e_{\Sc}^1\mathbb Z_p(O_{K,S}^{\times})_{\mathrm{tf}}
\right]
\]
is well defined. Therefore, by multiplicativity of indices,
\begin{align*}
\Big(&
e_{\Sc}^1\mathbb Z_p(O_{K,\Sc}^{\times})_{\mathrm{tf}}
:\;
e_{\Sc}^1\mathbb Z_p\varrho_{K/k,\Sc}
\Big)_{\mathbb Z_p}\\
=&
\left[
e_{\Sc}^1\mathbb Z_p(O_{K,\Sc}^{\times})_{\mathrm{tf}}
:
e_{\Sc}^1\mathbb Z_p(O_{K,S}^{\times})_{\mathrm{tf}}
\right]
\left(
e_{\Sc}^1\mathbb Z_p(O_{K,S}^{\times})_{\mathrm{tf}}
:
e_{\Sc}^1\mathbb Z_p\varrho_{K/k,S}
\right)_{\mathbb Z_p}\left(
e_{\Sc}^1\mathbb Z_p\varrho_{K/k,S}
:
e_{\Sc}^1\mathbb Z_p\varrho_{K/k,\Sc}
\right)_{\mathbb Z_p}.
\end{align*}
Substituting the preceding formula gives the result.
\qed
\end{proof}
As a consequence of Theorem \ref{Thm:Ben-1-5-general-form}, we obtain the
following criterion. Recall that we write 
\[
|\mu(K)(p)|=p^n.
\]
Also, let us suppose in what follows that $|\mathrm{Ram}(K/k)|\neq 1$, so that $|\Sc|\neq 2$ whenever $\Sc=S\cup \mathrm{Ram}(K/k)$.
\begin{theorem}
\label{Thm:RS-obstruction-divisibility}
The following assertions hold.

\begin{enumerate}
\item If every place in \(\Sc\setminus S\) does not decompose in
\(K/k\) and
\[
h_{k,S}(p)
\nmid
h_{K,S}(p)\cdot
p^{
\min(n,v_p(|G|))d_p(G)
+
2v_p(|G|)|\Sc\setminus S|
},
\]
then the rank-one Rubin--Stark conjecture does not hold for \(K/k\) and
\(\Sc\).

\item Assume that the ramified places in \(K/k\) do not decompose in
\(K/k\).
Then
\[
h_{k,S}(p)
\mid
h_{K,S}(p)\cdot
p^{
\min(n,v_p(|G|))d_p(G)
+
2v_p(|G|)|\mathrm{Ram}(K/k)|
}.
\]
\end{enumerate}
\end{theorem}

\begin{proof}
By Theorem \ref{Thm:Ben-1-5-general-form}, we have
\begin{align*}
\Big(
e_{\Sc}^1&\mathbb Z_p(O_{K,\Sc}^{\times})_{\mathrm{tf}}
:
e_{\Sc}^1\mathbb Z_p\varrho_{K/k,\Sc}
\Big)_{\mathbb Z_p}\\
&=
\frac{h_{K,S}}{h_{k,S}}
\big| H^1(G,\mathbb Z/p^n\mathbb Z(-1)) \big|
\left[
e_{\Sc}^1\mathbb Z_p(O_{K,\Sc}^{\times})_{\mathrm{tf}}
:
e_{\Sc}^1\mathbb Z_p(O_{K,S}^{\times})_{\mathrm{tf}}
\right]\left(
e_{\Sc}^1\mathbb Z_p\varrho_{K/k,S}
:
e_{\Sc}^1\mathbb Z_p\varrho_{K/k,\Sc}
\right)_{\mathbb Z_p}.
\end{align*}
By Corollary \ref{Cor:correction-factor-bound},
\[
v_p\left(\big| H^1(G,\mathbb Z/p^n\mathbb Z(-1)) \big|\right)
\leq
\min(n,v_p(|G|))d_p(G).
\]
Since every place in
\(\Sc\setminus S\) does not decompose in \(K/k\), we have on the one hand,
by Lemma \ref{Lemma_4_8}
    
   \[
v_p\Big(\left(
e_{\Sc}^1\mathbb Z_p\varrho_{K/k,S}
:
e_{\Sc}^1\mathbb Z_p\varrho_{K/k,\Sc}
\right)_{\mathbb Z_p}\Big)
=
\sum_{v\in\Sc\setminus S}v_p([G:I_v])\leq v_p(|G|)|\Sc\setminus S|
.
\]
On the other hand, Lemma \ref{Lem:unit-index-bound} shows that
\[
v_p\left(
\left[
e_{\Sc}^1\mathbb Z_p(O_{K,\Sc}^{\times})_{\mathrm{tf}}
:
e_{\Sc}^1\mathbb Z_p(O_{K,S}^{\times})_{\mathrm{tf}}
\right]
\right)
\leq
v_p(|G|)|\Sc\setminus S|.
\]
Therefore
\[
v_p\left(
\big| H^1(G,\mathbb Z/p^n\mathbb Z(-1)) \big|
\left[
e_{\Sc}^1\mathbb Z_p(O_{K,\Sc}^{\times})_{\mathrm{tf}}
:
e_{\Sc}^1\mathbb Z_p(O_{K,S}^{\times})_{\mathrm{tf}}
\right]
\left(
e_{\Sc}^1\mathbb Z_p\varrho_{K/k,S}
:
e_{\Sc}^1\mathbb Z_p\varrho_{K/k,\Sc}
\right)_{\mathbb Z_p}\right)
\]
\[
\leq
\min(n,v_p(|G|))d_p(G)
+
2v_p(|G|)|\Sc\setminus S|.
\]

Now assume that the rank-one Rubin--Stark conjecture holds for \(K/k\) and
\(\Sc\). By Corollary \ref{Cor:eS-index-well-defined}, we then have
\[
\left(
e_{\Sc}^1\mathbb Z_p(O_{K,\Sc}^{\times})_{\mathrm{tf}}
:
e_{\Sc}^1\mathbb Z_p\varrho_{K/k,\Sc}
\right)_{\mathbb Z_p}
\subseteq
\mathbb Z_p.
\]
Hence
\[
h_{k,S}(p)
\mid
h_{K,S}(p)\cdot
p^{
\min(n,v_p(|G|))d_p(G)
+
2v_p(|G|)|\Sc\setminus S|
}.
\]
This proves the contrapositive of \((1)\).
For \((2)\), take
\[
\Sc=S\cup \mathrm{Ram}(K/k).
\]
Since $\Sc\supseteq\mathrm{Ram}(K/k)$ (and $\Sc\supseteq\mathrm{Ram}(K/k)\cup S_\infty$ in the number field case), the rank-one Rubin--Stark conjecture for $K/k$ and $\Sc$ holds unconditionally for finite abelian extensions of function fields and for finite abelian extensions of number fields with $k=\mathbb{Q}$ or $k$ an imaginary quadratic number  field (see Remark \ref{REMA_3_2}). But this matches exactly our case since we suppose that $|S|=1$ and that $S\supseteq S_\infty$ in the number field case.
we obtain
\[
h_{k,S}(p)
\mid
h_{K,S}(p)\cdot
p^{
\min(n,v_p(|G|))d_p(G)
+
2v_p(|G|)|\mathrm{Ram}(K/k)|
}.
\]
This proves \((2)\).\qed
\end{proof}
\begin{corollary}\label{Case-homology-vanishes}
  Assume that either
\[p\nmid |G| \qquad \text{or}\qquad p\nmid w_k.
\]
Then the following assertions hold.
\begin{enumerate}
\item If the places in \(\Sc\setminus S\) do not decompose in
\(K/k\) and
\[
h_{k,S}(p)
\nmid
h_{K,S}(p)\cdot
p^{2 v_p(|G|)|\Sc\setminus S|
},
\]
then the rank-one Rubin--Stark conjecture does not hold for \(K/k\) and
\(\Sc\).

\item Assume that the ramified places in \(K/k\) do not decompose in
\(K/k\).
Then
\[
h_{k,S}(p)
\mid
h_{K,S}(p)\cdot
p^{
2v_p(|G|)|\mathrm{Ram}(K/k)|
}.
\]
\end{enumerate}
\end{corollary}
\begin{proof}
    This follows from the proof of Theorem \ref{Thm:RS-obstruction-divisibility} and the fact that $H^1(G, \mathbb{Z}/p^n\mathbb{Z}(-1))=0$ in this case by Corollary \ref{Cor:correction-factor-trivial}. \qed
\end{proof}
\begin{rem}
    Corollary \ref{Case-homology-vanishes} recovers and expands \cite[Theorem 3.7]{El26}.
\end{rem}
\subsection{The case $|S|=2$.} \label{subsection:S=2}
In this subsection, we study the case where our depletion set $S$ satisfies $|S|=2$. In particular, we suppose that $S$ contains a place $\mathfrak{v}$ which is totally split in $K/k$ and an additional place which does not decompose in $K/k$. In this context, as we have seen in \eqref{r-S-chi} and \eqref{e-S-1}, we have
\[
r_S(\chi)=1\text{ for all }\chi\in\hat{G},\qquad \text{and}\qquad e_S^1=1.
\]
Then, the generalized Stickelberger module $\varrho_{K/k,S}$, which reflects the arithmetic behavior of the first derivatives of the associated $L$-functions, simplifies as follows
\begin{proposition}\label{Index-simplifies-for-cardinal2}
    Suppose that $S$ is as above. Then
    \[
\left(
(O_{K,S}^{\times})_{\mathrm{tf}}:
\varrho_{K/k,S}
\right)_{\mathbb{Z}}
=
{h_{K,S}}.
\]
\end{proposition}
\begin{proof}
    Combine Proposition \ref{Prop:generalized-index-formula-rank-one} and Proposition \ref{Prop:XKS-annihilator-index-two-cases}($2$).\qed
\end{proof}
Now let \(\Sc\supseteq S\). We shall again cut out the rank-one
component by the idempotent \(e_{\Sc}^1\). Before doing this, we need a
simple description of the trivial component of \(X_{K,\Sc}\).

Let
\[
\mathrm{res}:Y_{K,\Sc}\longrightarrow Y_{k,\Sc}
\]
be the restriction map induced by sending a place \(w\) of \(K\) to the place of \(k\) below
it. We use the same notation for its \(\mathbb Q\)-linear extension.

\begin{lemma}
\label{Lem:trivial-component-XKS}
Put
\[
S':=\{\mathfrak v\},
\qquad
m:=|\Sc\setminus S'|.
\]
Assume that every place in \(\Sc\setminus S'\) does not decompose in
\(K/k\). Then \(\mathrm{res}\) induces an isomorphism
\[
\mathrm{res}:e_{\chi_0}X_{K,\Sc}
\xrightarrow{\;\sim\;}
X_{k,\Sc}.
\]
\end{lemma}

\begin{proof}
Choose a place \(w_0\) of \(K\) above \(\mathfrak v\), and write
\[
\Sc\setminus S'=\{v_1,\dots,v_m\}.
\]
Since each \(v_i\) does not decompose in \(K/k\), there is a unique place
\(w_i\) of \(K\) above \(v_i\). Hence
\[
Y_{K,\Sc}
\simeq
\mathbb Z[G]w_0\oplus \bigoplus_{i=1}^m\mathbb Z w_i.
\]
Under this identification,
\[
X_{K,\Sc}
=
\left\{
(a,n_1,\dots,n_m)\in \mathbb Z[G]\oplus\mathbb Z^m
\; ; \;
\mathrm{Aug}_G(a)+n_1+\cdots+n_m=0
\right\}.
\]
Moreover,
\[
X_{k,\Sc}
=
\left\{
b_0\mathfrak v+b_1v_1+\cdots+b_mv_m
\; ; \;
b_0+b_1+\cdots+b_m=0
\right\}.
\]

The idempotent \(e_{\chi_0}\) acts on \(\mathbb Q[G]w_0\) by
\[
e_{\chi_0}aw_0=\frac{\mathrm{Aug}_G(a)}{|G|}N_Gw_0,
\]
and acts trivially on the \(G\)-fixed summands \(\mathbb Z w_i\). Therefore
\[
e_{\chi_0}X_{K,\Sc}
=
\left\{
\frac{b_0}{|G|}N_Gw_0+b_1w_1+\cdots+b_mw_m
\; ; \;
b_i\in\mathbb Z,\;
b_0+b_1+\cdots+b_m=0
\right\}.
\]
Since
\[
\mathrm{res}\left(\frac{1}{|G|}N_Gw_0\right)=\mathfrak v
\qquad
\text{and}
\qquad
\mathrm{res}(w_i)=v_i,
\]
the restriction map sends this lattice bijectively onto \(X_{k,\Sc}\).
This proves the claim.\qed
\end{proof}

\begin{proposition}
\label{Prop:trivial-component-index}
Put
\[
S':=\{\mathfrak v\},
\qquad
m:=|\Sc\setminus S'|.
\]
Assume that every place in \(\Sc\setminus S'\) does not decompose in
\(K/k\), and suppose that \(m\geq 1\). Then
\[
\left(
e_{\chi_0}(O_{K,\Sc}^{\times})_{\mathrm{tf}}
:
e_{\chi_0}\varrho_{K/k,\Sc}
\right)_{\mathbb Z}
=
\frac{
h_{k,\Sc}^{m}R_{k,\Sc}^{m-1}
}{
\left[
(O_{k,\Sc}^{\times})_{\mathrm{tf}}
:
N_{K/k}(O_{K,\Sc}^{\times})_{\mathrm{tf}}
\right]
}
\mathbb Z.
\]
\end{proposition}

\begin{proof}
By Lemma \ref{Lem:trivial-component-XKS}, the map \(\mathrm{res}\) induces an
isomorphism
\[
\mathrm{res}:e_{\chi_0}X_{K,\Sc}\xrightarrow{\;\sim\;}X_{k,\Sc}.
\]
Hence, by functoriality of generalized indices applied to the isomorphism
\[
\mathrm{res}\circ \widetilde{\lambda}_{K,\Sc}:
e_{\chi_0}\mathbb C(O_{K,\Sc}^{\times})_{\mathrm{tf}}
\xrightarrow{\;\sim\;}
\mathbb C X_{k,\Sc},
\]
we get
\[
\left(
e_{\chi_0}(O_{K,\Sc}^{\times})_{\mathrm{tf}}
:
e_{\chi_0}\varrho_{K/k,\Sc}
\right)_{\mathbb Z}
=
\left(
\mathrm{res}\left(\widetilde{\lambda}_{K,\Sc}
\left(e_{\chi_0}(O_{K,\Sc}^{\times})_{\mathrm{tf}}\right)\right)
:
\mathrm{res}\left(\widetilde{\lambda}_{K,\Sc}
\left(e_{\chi_0}\varrho_{K/k,\Sc}\right)\right)
\right)_{\mathbb Z}.
\]

By the compatibility of the logarithmic regulator with restriction and norm
(see e.g., \cite[Lemma~5.4]{SaadMazigh}), one has
\[
\mathrm{res}\circ \widetilde{\lambda}_{K,\Sc}
=
\widetilde\lambda_{k,\Sc}\circ N_{K/k}
\]
on \((O_{K,\Sc}^{\times})_{\mathrm{tf}}\). Therefore
\[
\mathrm{res}\left(\widetilde{\lambda}_{K,\Sc}
\left(e_{\chi_0}(O_{K,\Sc}^{\times})_{\mathrm{tf}}\right)\right)
=
\widetilde\lambda_{k,\Sc}\left(N_{K/k}(O_{K,\Sc}^{\times})_{\mathrm{tf}}\right).
\]

It remains to compute the image of \(e_{\chi_0}\varrho_{K/k,\Sc}\). 
By definition,
\[
\widetilde{\lambda}_{K,\Sc}(\varrho_{K/k,\Sc})
=
\mathrm{Ann}_{\mathbb Z[G]}(\mu(K))\Theta_{K/k,\Sc}^{*}(0)X_{K,\Sc}.
\]
Multiplying by \(e_{\chi_0}\), and using the fact that the trivial component of
\(\Theta_{K/k,\Sc}^{*}(0)\) is \(\zeta_{k,\Sc}^{*}(0)e_{\chi_0}\), gives
\[
\widetilde{\lambda}_{K,\Sc}
\left(e_{\chi_0}\varrho_{K/k,\Sc}\right)
=
e_{\chi_0}\mathrm{Ann}_{\mathbb Z[G]}(\mu(K))\zeta_{k,\Sc}^{*}(0)X_{K,\Sc}.
\]
Applying \(\mathrm{res}\), we obtain
\[
\mathrm{res}\left(
\widetilde{\lambda}_{K,\Sc}
\left(e_{\chi_0}\varrho_{K/k,\Sc}\right)
\right)
=
\mathrm{Aug}_G(\mathrm{Ann}_{\mathbb Z[G]}(\mu(K)))\zeta_{k,\Sc}^{*}(0)X_{k,\Sc}.
\]
Moreover by (\ref{roots-of-units-augmentation}), we have
\[
\mathrm{Aug}_G(\mathrm{Ann}_{\mathbb Z[G]}(\mu(K)))
=
\mathrm{Ann}_{\mathbb Z}(\mu(k))
=
w_k\mathbb Z.
\]
Thus
\[
\mathrm{res}\left(
\widetilde{\lambda}_{K,\Sc}
\left(e_{\chi_0}\varrho_{K/k,\Sc}\right)
\right)
=
w_k\zeta_{k,\Sc}^{*}(0)X_{k,\Sc}.
\]
By the \(\Sc\)-class number formula,
\[
w_k\zeta_{k,\Sc}^{*}(0)
=
\pm h_{k,\Sc}R_{k,\Sc}.
\]
The sign does not affect generalized indices, so
\[
\mathrm{res}\left(
\widetilde{\lambda}_{K,\Sc}
\left(e_{\chi_0}\varrho_{K/k,\Sc}\right)
\right)
=
h_{k,\Sc}R_{k,\Sc}X_{k,\Sc}
\]
up to sign.
Consequently,
\[
\left(
e_{\chi_0}(O_{K,\Sc}^{\times})_{\mathrm{tf}}
:
e_{\chi_0}\varrho_{K/k,\Sc}
\right)_{\mathbb Z}
=
\left(
\widetilde\lambda_{k,\Sc}\left(N_{K/k}(O_{K,\Sc}^{\times})_{\mathrm{tf}}\right)
:
h_{k,\Sc}R_{k,\Sc}X_{k,\Sc}
\right)_{\mathbb Z}.
\]
By multiplicativity of generalized indices,
\[
\left(
\widetilde\lambda_{k,\Sc}\left(N_{K/k}(O_{K,\Sc}^{\times})_{\mathrm{tf}}\right)
:
h_{k,\Sc}R_{k,\Sc}X_{k,\Sc}
\right)_{\mathbb Z}
=
\left(
\widetilde\lambda_{k,\Sc}\left(N_{K/k}(O_{K,\Sc}^{\times})_{\mathrm{tf}}\right)
:
\widetilde\lambda_{k,\Sc}\left((O_{k,\Sc}^{\times})_{\mathrm{tf}}\right)
\right)_{\mathbb Z}
\]
\[
\cdot
\left(
\widetilde\lambda_{k,\Sc}\left((O_{k,\Sc}^{\times})_{\mathrm{tf}}\right)
:
X_{k,\Sc}
\right)_{\mathbb Z}
\cdot
\left(
X_{k,\Sc}
:
h_{k,\Sc}R_{k,\Sc}X_{k,\Sc}
\right)_{\mathbb Z}.
\]
The first factor is
\[
\left[
(O_{k,\Sc}^{\times})_{\mathrm{tf}}
:
N_{K/k}(O_{K,\Sc}^{\times})_{\mathrm{tf}}
\right]^{-1}\mathbb Z.
\]
By definition of the \(\Sc\)-regulator,
\[
\left(
X_{k,\Sc}:
\widetilde\lambda_{k,\Sc}\left((O_{k,\Sc}^{\times})_{\mathrm{tf}}\right)
\right)_{\mathbb Z}
=
R_{k,\Sc}\mathbb Z,
\]
and hence
\[
\left(
\widetilde\lambda_{k,\Sc}\left((O_{k,\Sc}^{\times})_{\mathrm{tf}}\right)
:
X_{k,\Sc}
\right)_{\mathbb Z}
=
R_{k,\Sc}^{-1}\mathbb Z.
\]
Finally, since
\[
\operatorname{rank}_{\mathbb Z}X_{k,\Sc}=m,
\]
one has
\[
\left(
X_{k,\Sc}:h_{k,\Sc}R_{k,\Sc}X_{k,\Sc}
\right)_{\mathbb Z}
=
(h_{k,\Sc}R_{k,\Sc})^m\mathbb Z.
\]
Combining these three factors gives
\[
\left(
e_{\chi_0}(O_{K,\Sc}^{\times})_{\mathrm{tf}}
:
e_{\chi_0}\varrho_{K/k,\Sc}
\right)_{\mathbb Z}
=
\frac{
h_{k,\Sc}^{m}R_{k,\Sc}^{m-1}
}{
\left[
(O_{k,\Sc}^{\times})_{\mathrm{tf}}
:
N_{K/k}(O_{K,\Sc}^{\times})_{\mathrm{tf}}
\right]
}
\mathbb Z.
\]\qed
\end{proof}

\begin{theorem}
\label{ClassNumberRelation-Cardinal2}
Assume that \(\Sc\supsetneq S\) and that every place in
\(\Sc\setminus S\) does not decompose in \(K/k\). Then
\[
p^{v_p(|G|)(|G|-1+2|\Sc\setminus S|)}
\frac{h_{K,S}}{h_{k,S}}
\big|\widehat H^0(G, (O_{K,S}^{\times})_{\mathrm{tf}})\big| \, \mathbb Z_p
\subseteq
\left(
e_{\Sc}^1\mathbb Z_p(O_{K,\Sc}^{\times})_{\mathrm{tf}}
:
e_{\Sc}^1\mathbb Z_p\varrho_{K/k,\Sc}
\right)_{\mathbb Z_p},
\]
where $\widehat H^0$ denotes the zeroth Tate cohomology group.
\end{theorem}

\begin{proof}
Since \(\Sc\supsetneq S\), we have \(|\Sc|>2\). Moreover, since
the places in \(\Sc\setminus S\) do not decompose in \(K/k\), one has
\[
r_{\Sc}(\chi)=1
\qquad
\text{for all }\chi\in\widehat G\setminus\{\chi_0\}
\]
by \eqref{r-S-chi}, and therefore
\[
e_{\Sc}^1=1-e_{\chi_0}.
\]
Put
\[
e:=e_{\Sc}^1,\qquad
M:=\mathbb Z_p(O_{K,S}^{\times})_{\mathrm{tf}},
\qquad
N:=\mathbb Z_p\varrho_{K/k,S}.
\]
We apply Lemma \ref{Lem:idempotent-index-comparison} with
\[
\mathcal O=\mathbb Z_p,\qquad E=\mathbb Q_p,\qquad
V=\mathbb Q_p(O_{K,S}^{\times})_{\mathrm{tf}}.
\]
Together with Proposition \ref{Index-simplifies-for-cardinal2}, this gives
\[
(eM:eN)_{\mathbb Z_p}
\bigl((1-e)M:(1-e)N\bigr)_{\mathbb Z_p}
=
p^{
\ell_{\mathbb Z_p}(eM/(eM\cap M))
-
\ell_{\mathbb Z_p}(eN/(eN\cap N))
}
h_{K,S}\mathbb Z_p.
\]

We now compute the \((1-e)\)-component. Since \(1-e=e_{\chi_0}\), Proposition
\ref{Prop:trivial-component-index}, applied with
\(m=1\), gives
\[
\bigl((1-e)M:(1-e)N\bigr)_{\mathbb Z_p}
=
\frac{h_{k,S}}
{
\left[
(O_{k,S}^{\times})_{\mathrm{tf}}
:
N_{K/k}(O_{K,S}^{\times})_{\mathrm{tf}}
\right]
}
\mathbb Z_p.
\]
By the definition of Tate cohomology,
\[
\left[
(O_{k,S}^{\times})_{\mathrm{tf}}
:
N_{K/k}(O_{K,S}^{\times})_{\mathrm{tf}}
\right]
=
\big| \widehat H^0(G, (O_{K,S}^{\times})_{\mathrm{tf}}) \big|,
\]
and we obtain
\[
(eM:eN)_{\mathbb Z_p}
=
p^{
\ell_{\mathbb Z_p}(eM/(eM\cap M))
-
\ell_{\mathbb Z_p}(eN/(eN\cap N))
}
\frac{h_{K,S}}{h_{k,S}}
\big| \widehat H^0(G, (O_{K,S}^{\times})_{\mathrm{tf}})\big| \, \mathbb Z_p.
\]

It remains to control the possible denominator coming from the idempotent. Since
\[
|G|e\in\mathbb Z_p[G]
\]
and
\[
\dim_{\mathbb Q_p}(eV)=|G|-1,
\]
Lemma \ref{Lem:idempotent-index-comparison} gives
\[
0\leq
\ell_{\mathbb Z_p}(eM/(eM\cap M))
\leq
v_p(|G|)(|G|-1)
\]
and
\[
0\leq
\ell_{\mathbb Z_p}(eN/(eN\cap N))
\leq
v_p(|G|)(|G|-1).
\]
Hence
\[
\ell_{\mathbb Z_p}(eM/(eM\cap M))
-
\ell_{\mathbb Z_p}(eN/(eN\cap N))
\leq
v_p(|G|)(|G|-1).
\]
Therefore
\[
p^{v_p(|G|)(|G|-1)}
\frac{h_{K,S}}{h_{k,S}}
\big| \widehat H^0(G, (O_{K,S}^{\times})_{\mathrm{tf}}) \big| \, \mathbb Z_p
\subseteq
(eM:eN)_{\mathbb Z_p}.
\]

Finally, by combining Lemma \ref{Lemma_4_8}, Lemma \ref{Lem:unit-index-bound} and using the fact that $v_p\left(
\prod_{v\in\Sc\setminus S}[G:I_v]
\right)\leq v_p(|G|)|\Sc\setminus S|$,
we get
\[
p^{2v_p(|G|)|\Sc\setminus S|}
(eM:eN)_{\mathbb Z_p}
\subseteq
\left(
e\mathbb Z_p(O_{K,\Sc}^{\times})_{\mathrm{tf}}
:
e\mathbb Z_p\varrho_{K/k,\Sc}
\right)_{\mathbb Z_p}.
\]
Combining this with the preceding inclusion gives
\[
p^{v_p(|G|)(|G|-1+2|\Sc\setminus S|)}
\frac{h_{K,S}}{h_{k,S}}
\big| \widehat H^0(G, (O_{K,S}^{\times})_{\mathrm{tf}}) \big| \, \mathbb Z_p
\subseteq
\left(
e_{\Sc}^1\mathbb Z_p(O_{K,\Sc}^{\times})_{\mathrm{tf}}
:
e_{\Sc}^1\mathbb Z_p\varrho_{K/k,\Sc}
\right)_{\mathbb Z_p}.
\]
This proves the result.\qed
\end{proof}
Next, we let 
\[
\Sc=S\cup\mathrm{Ram}(K/k).
\]
We assume that this latter is not $S$ and, in the number field case, that the rank-one Rubin--Stark
conjecture holds for \(K/k\) and \(\Sc\).

\begin{theorem}
\label{Thm:cardinality-two-final-divisibility}
Under these conditions, assume that the ramified places in \(K/k\) do not decompose in \(K/k\). Then the
following assertions hold.

\begin{enumerate}
\item If the norm map
\[
N_{K/k}:
\mathbb Z_p(O_{K,S}^{\times})_{\mathrm{tf}}
\longrightarrow
\mathbb Z_p(O_{k,S}^{\times})_{\mathrm{tf}}
\]
is surjective, then
\[
h_{k,S}(p)
\mid
h_{K,S}(p)\,
p^{v_p(|G|)\left(|G|-1+2|\mathrm{Ram}(K/k)|\right)}.
\]

\item If \(p\nmid |G|\), then
\[
h_{k,S}(p)\mid h_{K,S}(p).
\]
\end{enumerate}
\end{theorem}

\begin{proof}
We apply Theorem \ref{ClassNumberRelation-Cardinal2} with
\[
\Sc=S\cup \mathrm{Ram}(K/k).
\]
Since the ramified places are assumed not to decompose in \(K/k\), the hypotheses
of Theorem \ref{ClassNumberRelation-Cardinal2} are satisfied. Hence
\[
p^{v_p(|G|)\left(|G|-1+2|\Sc\setminus S|\right)}
\frac{h_{K,S}}{h_{k,S}}
\big| \widehat H^0(G, (O_{K,S}^{\times})_{\mathrm{tf}}) \big| \, \mathbb Z_p
\subseteq
\left(
e_{\Sc}^1\mathbb Z_p(O_{K,\Sc}^{\times})_{\mathrm{tf}}
:
e_{\Sc}^1\mathbb Z_p\varrho_{K/k,\Sc}
\right)_{\mathbb Z_p}.
\]
By the rank-one Rubin--Stark conjecture for \(K/k\) and \(\Sc\), and by
Corollary \ref{Cor:eS-index-well-defined}, we have
\[
\left(
e_{\Sc}^1\mathbb Z_p(O_{K,\Sc}^{\times})_{\mathrm{tf}}
:
e_{\Sc}^1\mathbb Z_p\varrho_{K/k,\Sc}
\right)_{\mathbb Z_p}
\subseteq
\mathbb Z_p.
\]
Therefore
\begin{equation} \label{auxiliary-valuation-inequality}
v_p\left(\frac{h_{K,S}}{h_{k,S}}\right)
+
v_p(|G|)\left(|G|-1+2|\Sc\setminus S|\right)
+
v_p\left(\big| \widehat H^0(G, (O_{K,S}^{\times})_{\mathrm{tf}}) \big| \right)
\geq 0.
\end{equation}

Assume now that the norm map
\(
N_{K/k}:
\mathbb Z_p(O_{K,S}^{\times})_{\mathrm{tf}}
\longrightarrow
\mathbb Z_p(O_{k,S}^{\times})_{\mathrm{tf}}
\)
is surjective. By the standard identity
\[
\left[
(O_{k,S}^{\times})_{\mathrm{tf}}
:
N_{K/k}(O_{K,S}^{\times})_{\mathrm{tf}}
\right]
=
\big| \widehat H^0(G, (O_{K,S}^{\times})_{\mathrm{tf}}) \big|,
\]
the \(p\)-part of \(\widehat H^0(G, (O_{K,S}^{\times})_{\mathrm{tf}})\) is trivial. Hence
\[
v_p\left(\frac{h_{K,S}}{h_{k,S}}\right)
+
v_p(|G|)\left(|G|-1+ 2 |\Sc\setminus S|\right)
\geq 0.
\]
Since
\[
|\Sc\setminus S|\leq |\mathrm{Ram}(K/k)|,
\]
we get
\[
h_{k,S}(p)
\mid
h_{K,S}(p)\,
p^{v_p(|G|)\left(|G|-1+ 2 |\mathrm{Ram}(K/k)|\right)}.
\]
This proves \((1)\).
If \(p\nmid |G|\), then \(v_p(|G|)=0\), and the \(p\)-part of
\(\widehat H^0(G, (O_{K,S}^{\times})_{\mathrm{tf}})\) is trivial because Tate cohomology is
annihilated by \(|G|\). Therefore the preceding valuation inequality \eqref{auxiliary-valuation-inequality} reduces to
\[
v_p\left(\frac{h_{K,S}}{h_{k,S}}\right)\geq 0,
\]
or equivalently
\[
h_{k,S}(p)\mid h_{K,S}(p).
\]
This proves \((2)\).\qed
\end{proof}

\subsection{The case \( |S|\geq 3 \)} \label{subsection:S>=3}

We finish by indicating what changes when \(|S|\geq 3\). Suppose that \(S\)
contains one place \(\mathfrak v\) which splits completely in \(K/k\), and that
the remaining places of \(S\) do not decompose in \(K/k\). Put
\[
S':=\{\mathfrak v\},
\qquad
m:=|S\setminus S'|.
\]
For the trivial character, one has
\[
r_S(\chi_0)=|S|-1=m.
\]
Thus, when \(m\geq 2\), the trivial component is no longer part of the rank-one (or rank-zero)
idempotent. Equivalently, the corresponding Rubin--Stark element belongs to a
higher exterior bidual of the group of \(S\)-units.

The same phenomenon appears on the divisor side: \(X_{K,S}\) is an extension of
\(\Delta G\) by a trivial \(G\)-module of rank \(m\). The following computation
shows how this produces extra factors involving the class number and regulator
of \(k\).
\begin{proposition}
\label{Prop:XKS-annihilator-index-general-S}
Assume that \(S\) contains a unique place \(\mathfrak v\) which splits completely
in \(K/k\). Put
\[
S':=\{\mathfrak v\}.
%\qquad
%m:=|S\setminus S'|.
\]
Assume moreover that \(m\geq 1\) and that every place in \(S\setminus S'\) does
not decompose in \(K/k\). Then
\[
\left[
X_{K,S}:
\mathrm{Ann}_{\mathbb Z[G]}(\mu(K))X_{K,S}
\right]
=
w_Kw_k^{|S\setminus S'|-1}.
\]
\end{proposition}

\begin{proof}
Choose a place \(w_0\) of \(K\) above \(\mathfrak v\), and write
\[
S\setminus S'=\{v_1,\dots,v_m\}.
\]
Since each \(v_i\) does not decompose in \(K/k\), there is a unique place
\(w_i\) of \(K\) above \(v_i\), and \(G\) acts trivially on \(w_i\). Hence
\[
Y_{K,S}
\simeq
\mathbb Z[G]w_0\oplus \bigoplus_{i=1}^m\mathbb Z w_i.
\]
Under this identification,
\[
X_{K,S}
=
\left\{
(a,n_1,\dots,n_m)\in \mathbb Z[G]\oplus\mathbb Z^m
\; ; \;
\mathrm{Aug}_G(a)+n_1+\cdots+n_m=0
\right\}.
\]
Thus there is an exact sequence of \(\mathbb Z[G]\)-modules
\[
0
\longrightarrow
\Delta G
\longrightarrow
X_{K,S}
\longrightarrow
\mathbb Z^m
\longrightarrow
0,
\]
where \(G\) acts trivially on \(\mathbb Z^m\).
Since \(G\) acts trivially on \(\mathbb Z^m\), the image of
\[
\mathrm{Ann}_{\mathbb Z[G]}(\mu(K))X_{K,S}
\]
in \(\mathbb Z^m\) is
\[
\mathrm{Aug}_G\bigl(
\mathrm{Ann}_{\mathbb Z[G]}(\mu(K))
\bigr)\mathbb Z^m.
\]
By \eqref{Ann-X-cap-Delta-G}, we have
\[
\mathrm{Ann}_{\mathbb Z[G]}(\mu(K))X_{K,S}\cap \Delta G
=
\mathrm{Ann}_{\mathbb Z[G]}(\mu(K))\cap \Delta G.
\]
Consequently, we have an exact sequence
\[
0
\longrightarrow
\frac{\Delta G}
{\mathrm{Ann}_{\mathbb Z[G]}(\mu(K))\cap \Delta G}
\longrightarrow
\frac{X_{K,S}}
{\mathrm{Ann}_{\mathbb Z[G]}(\mu(K))X_{K,S}}
\longrightarrow
\frac{\mathbb Z^m}
{\mathrm{Aug}_G\bigl(
\mathrm{Ann}_{\mathbb Z[G]}(\mu(K))
\bigr)\mathbb Z^m}
\longrightarrow
0.
\]
Therefore
\[
\left[
X_{K,S}:
\mathrm{Ann}_{\mathbb Z[G]}(\mu(K))X_{K,S}
\right]
=
\left[
\Delta G:
\mathrm{Ann}_{\mathbb Z[G]}(\mu(K))\cap \Delta G
\right]
\left[
\mathbb Z:
\mathrm{Aug}_G\bigl(
\mathrm{Ann}_{\mathbb Z[G]}(\mu(K))
\bigr)
\right]^m.
\]

By our  standard augmentation computation in (\ref{roots-of-units-augmentation}), we have
\[
\mathrm{Aug}_G\bigl(
\mathrm{Ann}_{\mathbb Z[G]}(\mu(K))
\bigr)
=
\mathrm{Ann}_{\mathbb Z}(\mu(k))
=
w_k\mathbb Z.
\]
Hence
\[
\left[
\mathbb Z:
\mathrm{Aug}_G\bigl(
\mathrm{Ann}_{\mathbb Z[G]}(\mu(K))
\bigr)
\right]
=
w_k.
\]

On the other hand,
\[
\mathbb Z[G]/\mathrm{Ann}_{\mathbb Z[G]}(\mu(K))
\simeq
\mu(K),
\]
and therefore
\[
\left[
\mathbb Z[G]:
\mathrm{Ann}_{\mathbb Z[G]}(\mu(K))
\right]
=
w_K.
\]
The exact sequence
\[
0
\longrightarrow
\frac{\Delta G}
{\mathrm{Ann}_{\mathbb Z[G]}(\mu(K))\cap \Delta G}
\longrightarrow
\frac{\mathbb Z[G]}
{\mathrm{Ann}_{\mathbb Z[G]}(\mu(K))}
\longrightarrow
\frac{\mathbb Z}
{\mathrm{Aug}_G\bigl(
\mathrm{Ann}_{\mathbb Z[G]}(\mu(K))
\bigr)}
\longrightarrow
0
\]
gives
\[
\left[
\Delta G:
\mathrm{Ann}_{\mathbb Z[G]}(\mu(K))\cap \Delta G
\right]
=
\frac{w_K}{w_k}.
\]
Substituting this into the preceding formula yields
\[
\left[
X_{K,S}:
\mathrm{Ann}_{\mathbb Z[G]}(\mu(K))X_{K,S}
\right]
=
\frac{w_K}{w_k}w_k^m
=
w_Kw_k^{m-1}.
\]
This proves the proposition.\qed
\end{proof}

\begin{theorem}
\label{Thm:generalized-index-large-S}
Assume  that \(|S\setminus S'|\geq 1\) and that every place in \(S\setminus S'\) does
not decompose in \(K/k\). Then
\[
\left(
(O_{K,S}^{\times})_{\mathrm{tf}}
:
\varrho_{K/k,S}
\right)_{\mathbb Z}
=
h_{K,S}h_{k,S}^{|S\setminus S'|-1}R_{k,S}^{|S\setminus S'|-1}\mathbb Z.
\]
\end{theorem}

\begin{proof}
Combine Proposition \ref{Prop:generalized-index-formula-rank-one} and Proposition \ref{Prop:XKS-annihilator-index-general-S}.
\qed
\end{proof}
\section{Conclusion and perspectives}
\label{section:conclusion}

The computations above show a clear difference between the cases \(|S|=1,2\)
and the higher-cardinality situation. In the first two cases, the generalized
Stickelberger module leads to relatively direct class number relations. By
contrast, when \(|S|\geq 3\), the trivial character has order of vanishing at
least two. As a consequence, the trivial component is naturally related to
higher Rubin--Stark elements, and the index formula of Theorem
\ref{Thm:generalized-index-large-S} involves higher powers of invariants of the
base field \(k\), namely
\[
h_{k,S}^{|S\setminus S'|-1}R_{k,S}^{|S\setminus S'|-1}.
\]
These terms are arithmetically meaningful, but they make the resulting formula
less directly suited to extracting simple divisibility relations between the
class numbers of \(K\) and \(k\).

This suggests that, beyond the rank-one range, one should replace the rank-one
generalized Stickelberger module by a higher-rank analogue. A possible direction
would be to define a module \(\varrho_{K/k,S}^{(r)}\) as the inverse image,
under the appropriate higher regulator map, of a lattice of the form
\[
\Theta_{K/k,S}^{*}(0)\,
\mathrm{Ann}_{\mathbb Z[G]}(\mu(K))\,
\bigcap_{\mathbb Z[G]}^r X_{K,S},
\]
where \(\bigcap_{\mathbb Z[G]}^r\) denotes the exterior bidual. Such a
construction would be more compatible with higher Rubin--Stark elements. It
remains to determine under which additional hypotheses on \(S\) and on the
extension \(K/k\) this higher-rank construction yields effective class number
relations.

\end{document}